\documentclass[a4paper, 12pt, reqno]{amsart}
\usepackage[left=3cm,right=3cm]{geometry}

\usepackage{qmod}

\usepackage{comments}
\newComments\AK{AK}{blue}
\newComments\JG{JG}{orange}
\newComments\RO{R\'{O}}{red}

\usepackage[english]{babel}

\usepackage{amsthm, amssymb, amsfonts, amsmath, mathrsfs}
\usepackage[all,cmtip]{xy}

\usepackage{tikz,tikz-cd}

\usepackage{caption}
\usepackage{url}
\usepackage{hyperref}

\usepackage{xcolor}

\usepackage{array}
\newcommand{\PreserveBackslash}[1]{\let\temp=\\#1\let\\=\temp}
\newcolumntype{C}[1]{>{\PreserveBackslash\centering}p{#1}}
\newcolumntype{R}[1]{>{\PreserveBackslash\raggedleft}p{#1}}
\newcolumntype{L}[1]{>{\PreserveBackslash\raggedright}p{#1}}

\newcounter{stepcounter}
\theoremstyle{plain}

\newtheorem{thm}{Theorem}[section]

\newtheorem{lem}[thm]{Lemma}
\newtheorem{prop}[thm]{Proposition}
\newtheorem{cor}[thm]{Corollary}

\newtheorem{defn}[thm]{Definition}

\theoremstyle{definition}

\newtheorem{eg}[thm]{Example}
\newtheorem{rem}[thm]{Remark}
\newtheorem{remark}[thm]{Remark}

\date{}

\newcommand\bit{\begin{itemize}}
\newcommand\eit{\end{itemize}}
\newcommand\bet{\begin{enumerate}}
\newcommand\eet{\end{enumerate}}
\newcommand\ed{\end{document}}

\DeclareFontFamily{U}{mathx}{\hyphenchar\font45}
\DeclareFontShape{U}{mathx}{m}{n}{
      <5> <6> <7> <8> <9> <10>
      <10.95> <12> <14.4> <17.28> <20.74> <24.88>
      mathx10
      }{}
\DeclareSymbolFont{mathx}{U}{mathx}{m}{n}
\DeclareFontSubstitution{U}{mathx}{m}{n}
\DeclareMathAccent{\widecheck}{0}{mathx}{"71}
\DeclareMathAccent{\wideparen}{0}{mathx}{"75}

\newcommand\Om{\Omega}

\newcommand\bN{{\mathbb N}}

\newcommand\fh{{\mathfrak h}}

\newcommand\Aut{\mathrm{Aut}}

\newcommand\co{\mathrm{co}}

\newcommand\exd{\mathrm{d}}

\newcommand\End{\mathrm{End}}

\newcommand\Hol{\mathrm{Hol}}

\newcommand\id{\mathrm{id}}

\newcommand\sseq{\subseteq}

\def\qbinom#1#2{\ensuremath{\left[\kern-.3em\left[\genfrac{}{}{0pt}{}{#1}{#2}\right]\kern-.3em\right]_q}}

\newcommand\ol{\overline}

\DeclareMathOperator{\ev}{ev}
\DeclareMathOperator{\coev}{coev}

\newcommand{\Ch}{\mathrm{Ch}}
\newcommand{\op}{\mathrm{op}}

\usepackage[english]{babel}

\newcommand{\fl}{\mathfrak{l}}

\DeclareMathOperator{\Hom}{Hom}

\newcommand{\cO}{\mathcal{O}}

\usepackage{mathtools}

\DeclareMathOperator{\Span}{span}
\DeclareMathOperator{\coT}{coT}

\author[J. Bhowmick]{Jyotishman Bhowmick}
\address{Stat Math Unit, Indian Statistical Institute, 203 B.T. Road, Kolkata 700108, India}
\email{jyotishmanb@gmail.com}

\author[B. Ghosh]{Bappa Ghosh}
\address{Stat Math Unit, Indian Statistical Institute, 203 B.T. Road, Kolkata 700108, India}
\email{bappa0697@gmail.com}

\author[A. O. Krutov]{Andrey O. Krutov}
\address{Mathematical Institute of Charles University, Sokolovsk\'a 83, Prague, Czech Republic}
\email{andrey.krutov@matfyz.cuni.cz}

\author[R. \'O Buachalla]{R\'eamonn \'O Buachalla}
\address{Mathematical Institute of Charles University, Sokolovsk\'a 83, Prague, Czech Republic}
\email{obuachalla@karlin.mff.cuni.cz}

\title{Levi-Civita connection on the irreducible quantum flag manifolds}

\thanks{
B.G. is supported by CSIR-SPM Fellowship through the file number SPM09/0093(13862)/2022-EMR-I.
A.O.K. was supported by the GA\v{C}R projects 24-10887S.
A.O.K. and R.\'OB. was supported by and HORIZON-MSCA-2022-SE-01-01 CaLIGOLA.
R.\'OB. is supported by the GA\v{C}R/NCN grant \emph{Quantum Geometric Representation Theory and Noncommutative Fibrations} 24-11728K.
This article is based upon work from COST Action CaLISTA CA21109 supported by COST (European Cooperation in Science and Technology, \url{www.cost.eu}).
}


\begin{document}

\begin{abstract}
We classify covariant metrics (in the sense of Beggs and Majid) on a~class of quantum homogeneous spaces. In particular, our classification implies the existence of a~unique (up to scalar)  quantum symmetric covariant metric on the Heckenberger--Kolb calculi for the quantized irreducible flag manifolds. Moreover, we prove the existence and uniqueness of Levi-Civita connection for any real covariant metric for the Heckenberger--Kolb calculi. This generalizes Matassa's result (\cite{matassalevicivita}) for the quantum projective spaces.
\end{abstract}

\maketitle

\section{Introduction}

Since the inception of noncommutative geometry, the notions of metrics, connections and their compatibility have been studied by many mathematicians. In the spectral triple framework of Connes (\cite{Connes}), a~candidate for a~metric on the space of one-forms of the spectral triple was proposed in \cite{FGR}. The authors of \cite{FGR} also formulated a~compatibility of a~connection with a~given metric. Existence and uniqueness of Levi-Civita connections for  metrics on the space of one-forms of  certain classes of spectral triples have been proved. For example, see  \cite{article1} or the more recent work \cite{meslandrennie1} where the existence result is extended to any $\theta$-deformation of a~compact Riemannian manifold.  

On the other hand, the question of existence of metrics and compatibility of a~connection with a~given metric can be studied in the more generalized set-up of a~differential calculus on an algebra. There have been a~lot of approaches to address these questions for a~variety of examples.  Existence and uniqueness of Levi-Civita connections for bilinear metrics on a~certain class of calculi by deriving a~Koszul type formula was achieved in  \cite{KoszulJDG}. For similar results in the context of more general classes of metric for tame differential calculi, quantum groups and $\theta$-deformed even spheres, we refer to \cite{article7},  \cite{aschieriweber}, \cite{atiyahseq} and references therein.

In this article, we study metrics and connections on a~class of quantum homogeneous spaces and the approach we follow here is the one presented in the monograph \cite{BeggsMajid:Leabh} of Beggs and Majid. Given a~differential calculus on an algebra $B$, this notion of metric is given by a~pair $ (g, (~, ~)) $ and  is a~choice of self-dual structure (in the category of $B$-bimodules) on the space of one-forms of the differential calculus. This makes it particularly suited to the study of covariant differential calculi over quantum homogeneous spaces. 

Suppose $A$ and $H$ are Hopf $\ast$-algebras where $H$ is a~compact quantum group algebra, $\pi: A \rightarrow
H$ a surjective Hopf $\ast$-algebra map and let  $B $ be the associated  quantum homogeneous space in the
sense of Definition \ref{ Quantum homogeneous space }. Moreover, assume that $B$ is equipped with an
$A$-covariant $\ast$-differential calculus~$\Omega^1(B)$.  We start by constructing covariant metrics
on~$\Omega^1(B)$ in the sense of \cite{BeggsMajid:Leabh} and then,  by a~categorical equivalence due to Takeuchi (Theorem \ref{3rdapril241}), we classify all such metrics. It turns out that the class of covariant metrics are parametrized by a~product of general linear groups. We refer to Theorem \ref{8thmay24jb5} for the precise statement. Moreover, when the calculus is equipped with a~covariant complex structure, we construct  a~family of covariant real metrics (see Definition~\ref{3rdjuly244}). 

After having a~good understanding on the space of covariant metrics on the above-mentioned class of quantum homogeneous spaces, we are now in a~position to ask for the existence and uniqueness of Levi-Civita connections for these metrics. We concentrate on the quantum homogeneous spaces coming from the Drinfeld--Jimbo algebras. For the case of the standard quantum $2$-sphere, this was settled by Majid in \cite{Maj}. More recently, it was extended to the case of all quantum projective spaces for a~generalized Fubini-Study metric by Matassa in \cite{matassalevicivita}. Our second goal in this article is to extend this theorem to any irreducible quantum flag manifold for a~class of covariant metrics. 

 Suppose $G$ is a~compact connected semisimple Lie group and $L_S$ a~choice of Levi-subalgebra for a~certain choice of a~subset of simple roots of the complexified Lie algebra $\mathfrak{g}$ of $G$ (see Section \ref{4thaugust244}). Then for all $0 < q < 1$, one defines the quantized  flag manifolds denoted by    $\mathcal{O}_q (G/L_S)$, which are  quantum homogeneous spaces  of $ \mathcal{O}_q (G), $ the latter being the coordinate algebra corresponding to the Drinfeld--Jimbo's universal enveloping algebra $ U_q (\mathfrak{g}). $

In the irreducible case, that is to say, for those quantum flag manifolds whose classical limit is a~symmetric space, the classical de Rham complex admits an essentially unique $q$-deformation which remembers a~large part of the classical flag manifold K\"ahler geometry. This is the celebrated Heckenberger--Kolb differential calculus (\cite{HK,HKdR}). 

Now suppose $  (g, (~, ~))   $ is  an $ \mathcal{O}_q (G) $-covariant metric on  $\mathcal{O}_q (G/L_S)$. Our strategy to construct the Levi-Civita connection for the pair  $  (g, (~, ~))   $ is to use the theory of Chern connections for holomorphic bimodules developed by Beggs and Majid in \cite{BeggsMajidChern}. The Heckenberger--Kolb calculus $ (\Omega^\bullet_q (G/L_S), \wedge, d) $ admits a~complex structure (see \cite{HK}, \cite{HKdR}, and~\cite{MarcoConj}) such that  $\Omega^{(1,0)}_q (G/L_S) $ is a~holomorphic $\mathcal{O}_q (G/L_S) $-bimodule. Similarly, $\Omega^{(0,1)}_q (G/L_S)$ is a~holomorphic bimodule with respect to the opposite complex structure. Thus, for Hermitian metrics (see Definition \ref{4thdec231}) $\mathscr{H}_1$ and $\mathscr{H}_2$ on $\Omega^{(1,0)}_q (G/L_S)$ and $\Omega^{(0,1)}_q (G/L_S)$ respectively, we have Chern connections $\nabla_{\Ch}$ and $\nabla_{\Ch, \op}$ respectively. Thus, we have a~connection
\begin{equation} \label{14thaug241}
 \nabla = \nabla_{\Ch} + \nabla_{\Ch, \op}
 \end{equation}
 on $ \Omega^{1}_q (G/L_S) = \Omega^{(1,0)}_q (G/L_S) \oplus \Omega^{(0,1)}_q (G/L_S)$. 
 Therefore, we need to connect Hermitian metrics on  $\Omega^{(1,0)}_q (G/L_S)$ and $\Omega^{(0,1)}_q (G/L_S)$ with metrics on $\Omega^{1}_q (G/L_S)$. This is where we need to restrict our attention to the case where $ (g, (~, ~))$ is a~real metric. 

We observe that for any $\ast$-differential calculus  $ (\Omega^\bullet, \wedge, d), $ real metrics  are in one to one correspondence with Hermitian metrics on $\Omega^1$.  For $\ast$-differential calculi on quantum homogeneous spaces equipped with complex structures such that $\Omega^{(1,0)}$ and $\Omega^{(0,1)}$ are simple and non-isomorphic objects in a~certain monoidal category, we prove that the Hermitian metric $H$ constructed from a~real metric $ (g, (~, ~)) $ is of the form $\mathscr{H} = \mathscr{H}_1 \oplus \mathscr{H}_2$, where $\mathscr{H}_1$ and $\mathscr{H}_2$ are Hermitian metrics on $\Omega^{(1,0)}$ and $\Omega^{(0,1)}$ respectively. In particular, this is true for the Heckenberger--Kolb calculi. 

Summarizing, starting from a~covariant real metric $(g, (~ , ~))$ on $\mathcal{O}_q (G/L_S), $ we have Chern connections $\nabla_{\Ch}  $ and $\nabla_{\Ch, \op}$
so that the connection $\nabla$ defined in \eqref{14thaug241} exists. In order to make sense of the metric-compatibility condition, we need to prove that $\nabla$ is a~bimodule connection. This follows from the fact that the Heckenberger--Kolb calculus is factorizable. Finally, the vanishing of the torsion of $\nabla$ and the compatibility of $\nabla$ with $(g, (~ , ~))$ are established using a~central element $Z \in U_q (\mathfrak{l}_S), $ where $Z$ is the natural analogue of the generator of the center of $\mathfrak{l}_S$.

\subsection*{Summary of the Paper}

The article is organized as follows: in Section \ref{10thaugust241}, we recall some necessary preliminaries about differential calculi, metrics (in the sense of Beggs and Majid), connections and their compatibility with metrics, bar categories and noncommutative complex structures.

In Section \ref{10thaugust242},  we classify covariant metrics on quantum homogeneous spaces of the form $B = A^{\co(H)}, $ where $A$ is a~Hopf $\ast$-algebra and $H$ a~compact quantum group algebra. When the quantum homogeneous space is equipped with a~covariant complex structure, we construct a~family of real covariant metrics on them. 

In Section \ref{10thaugust243}, we establish a~one to one correspondence between real covariant metrics and Hermitian metrics on the space of one-forms. When the calculus is endowed with a~complex structure such that $\Omega^{(1,0)}$ and $\Omega^{(0,1)}$ are simple and non-isomorphic, it turns out that the Hermitian metric constructed from a~real covariant metric splits into a~direct sum of Hermitian metrics on $\Omega^{(1,0)}$ and $\Omega^{(0,1)}$.   

In Section \ref{4thaugust243}, we show that for a~covariant factorizable complex structure, the Chern connection associated to a~covariant Hermitian metric is always a~covariant bimodule connection. 

In Section \ref{4thaugust244}, we recall some necessary details about Drinfeld--Jimbo quantum groups, irreducible quantum flag manifolds $\mathcal{O}_q(G/L_S)$, and the Heckenberger--Kolb calculi $\Omega^{\bullet}_q(G/L_S)$. We then apply the results of Section \ref{10thaugust242}  to classify  the covariant metrics on the Heckenberger--Kolb calculi. Moreover, we show that there exists a~unique (up to scalar) real covariant quantum symmetric metric on these calculi. Furthermore, we establish the existence and uniqueness of Levi-Civita connections for any real covariant metric. 

The results of Section \ref{4thaugust243} use the general fact that the Chern connection for a~covariant metric  on any covariant holomorphic bimodule is actually covariant. This is proved in Section \ref{1staugust241}. Finally, we collect some results in an appendix at the very end of the article. 

\subsection*{Acknowledgments}
J.B and B.G. would like to thank the organizers of the workshop ``Quantum groups and noncommutative geometry in
Prague, 2023" where this work was initiated.
We would like to thank Julien Bichon for discussions related to Lemma~\ref{22ndapril243}.

\section{Preliminaries} \label{10thaugust241}

Throughout this article, our ground field will be $\mathbb{C}$. All algebras are assumed to be unital. All unadorned tensor products are over $\mathbb{C}$.

Let $ (\mathcal{C}, \otimes) $ be an monoidal category with unit object $1_{\mathcal{C}}$. An object $M$ in the category ${\mathcal{C}}$ is said to have a~right dual if there exists an object $\prescript{\ast}{}{M}$ in ${\mathcal{C}}$ and  morphisms $\ev: M\otimes \prescript{\ast}{}{M} \to 1_{\mathcal{C}}$, $\coev: 1_{\mathcal{C}}\to \prescript{*}{}{M}\otimes M$ such that the equations
  \begin{align}\label{27thnov231}
      (\ev\otimes \id_M) (\id_M\otimes  \coev) =\id_M, \quad (\id_{\prescript{\ast}{}{M}}\otimes\ev)(\coev\otimes \id_{\prescript{\ast}{}{M}})  = \id_{\prescript{\ast}{}{M}}
  \end{align}
are satisfied.

Let $B$ be an~algebra. The category of all $B$-bimodules will be denoted by $\qMod{}{B}{}{B}$.
For a~$B$-bimodule $M$, we will denote the set of all left $B$-linear maps from~$M$ to~$B$ by the symbol $\prescript{}{B}{\Hom(M,B)}$. Then $\prescript{}{B}{\Hom(M,B)}$ is an~object of $\qMod{}{B}{}{B}$ via the formulas
\begin{align} \label{19thdec234}
    (b\cdot f)(m)=f(m\cdot b) \quad (f \cdot b)(m) = f(m)b,
\end{align}
where $f \in \prescript{}{B}{\Hom(M,B)}, m\in M $ and $ b \in B$.

We recall that if a~$B$-bimodule $M$ is finitely generated and projective as a~left $B$-module, then for some $n \in \mathbb{N}$, there are elements $e^i\in M, e_i \in \prescript{}{B}{\Hom(M,B)}$ for all $1\le i\le n$,  such that: 
\begin{align} \label{27thnov232}
m =\sum_i e_i(m) e^i ~ \text{and} ~  f =\sum_i e_i f(e^i) 
\end{align}
for all $ m\in M $ and for all $ f\in \prescript{}{B}{\Hom(M,B)}$. In this case, $\{ e^i, e_i: i = 1, \cdots n \}$ is called a~dual basis for $M$. 
 
We consider the maps
\[
  \coev: B\to \prescript{}{B}{\Hom(M,B)}\otimes_B M\quad\text{and}\quad
  \ev: M\otimes_B \prescript{}{B}{\Hom(M,B)} \to B
\]
defined as
\begin{align} \label{evcoev}
  \coev(1)=\sum_i e_i \otimes_B e^i \quad\text{and}\quad
  \ev(m\otimes_Bf)=f(m).   
\end{align}
Then it is well-known that $\ev$ and $\coev$ are morphisms of the category $\qMod{}{B}{}{B}$ and satisfy the equations~\eqref{27thnov231} and hence $\prescript{}{B}{\Hom(M,B)}$ is a~right dual of $M$ in the category $\qMod{}{B}{}{B}$.

Let us make a~note of the following well-known categorical result which essentially follows from~\cite[Proposition 2.10.8]{EtingofTensorCat}.

\begin{lem}\label{22ndapril243}
    Let $V$ be an object in a~monoidal category $(\mathcal{C}, \otimes, 1_\mathcal{C})$ with a~fixed right dual $(\prescript{*}{}{V}, \ev, \coev)$. Consider the set $X$ consisting of 
    all pairs $ (\ev^{\prime}, \coev^{\prime}) $ such that $(\prescript{*}{}{V}, \ev^{\prime}, \coev^{\prime})$ is a~right dual of $V$. 
    Then the group $\Aut(V)$ acts on $X$ freely and     transitively.
\end{lem}
 
More concretely, if  $V$ is an object in a~monoidal category $\mathcal{C}$ and  $\tau: V\to V$  an automorphism in $\mathcal{C}$, then the morphisms $\ev_\tau $ and $ \coev_\tau$ defined respectively by
\begin{equation} \label{21stmay24}
  \ev_\tau:= \ev\circ (\tau^{-1}\otimes\id)\ \text{and}\ \coev_\tau:= (\id\otimes \tau)\circ\coev 
\end{equation}
satisfy the evaluation-coevaluation equations. The action of $\Aut (V) $ on $X$ is given by
$$ \tau. (\ev, \coev) = (\ev_{\tau}, \coev_{\tau}). $$
    
\begin{eg} \label{eg:irreducible}
Let $H$ be a~Hopf algebra, and $V$ an irreducible left $H$-comodule $V$. Let us assume that $V$ admits a~right
dual ${}^\ast V$ (and $V$ is also a~right dual of~${}^\ast V$). Consider now the tensor product
\begin{align} \label{eqn:tensor.product}
({}^\ast V \oplus  V) \otimes ({}^\ast V \oplus  V) \simeq ({}^\ast V \otimes V) \oplus ({}^\ast V \otimes {}^\ast V) \oplus (V \otimes V) \oplus (V \otimes {}^\ast V).
\end{align}
Let $\ev_V: V \otimes {}^\ast V  \to \mathbb{C}$ be a~choice of evaluation map for $V$, and denote the uniquely
determined right coevaluation map by $\coev_V$.
Moreover, choose an~evaluation map $\ev_{{}^\ast V}: {}^\ast V \otimes V   \to \mathbb{C}$, and denote the uniquely
determined coevaluation map by $\coev_{{}^\ast V}$.
Then a~self-dual structure for $V \oplus {}^\ast V$ is given by 
\begin{align*}
\left(\ev := \ev_{{}^\ast V} \oplus ~ 0 \oplus ~ 0 \oplus \ev_{V}, \, \coev := \coev_{V} \oplus ~ 0 \oplus ~ 0 \oplus \coev_{{}^\ast V}\right)
\end{align*}

Let us now assume that $V$ is not isomorphic to ${}^\ast V$. Then by Schur's lemma for comodules (note that the case
of super Hopf algebras is more delicate; for example, see~\cite{NIS2}) we get the automorphism group of $V
\oplus {}^\ast V$ is $\mathrm{GL}_1(\mathbb{C}) \times \mathrm{GL}_1(\mathbb{C})$, which acts on its space of self-dual structures freely and transitively. Explicitly, for any automorphism $\tau = (\lambda_1,\lambda_2)$, where $\lambda_1, \lambda_2 \in \mathbb{C}^{\times}$, we see that
\begin{align*}
\left(\lambda_1^{-1} \ev_{V} + \lambda_2^{-1} \ev_{{}^\ast V}, \, \lambda_1 \coev_{V} + \lambda_2 \coev_{{}^\ast V} \right)
\end{align*}
is the associated twisted self-dual structure.

\end{eg}

\subsection{Hopf-algebras and relative Hopf-modules} \label{6thjune241}

Let $A$ be a Hopf-algebra with coproduct $\Delta: A\to A\otimes A$, counit $\epsilon: A\to \mathbb{C}$ and antipode $S:A\to A$. The antipode is always assumed to be bijective and moreover, we will use Sweedler notation to write 
\[
  \Delta (a) = a_{(1)} \otimes a_{(2)}.
\]
The symbol $\lMod{A}{}$ will denote the category of  left $A$ comodules. If $(M, \prescript{M}{}{\delta})$ is an object in $\lMod{A}{}$, we will often use  Sweedler notation for the coaction $\prescript{M}{}{\delta}:$
$$ \prescript{M}{}{\delta(m)}=\sum m_{(-1)}\otimes m_{(0)}. $$  

For a~Hopf-algebra $A$ and a~left $A$-comodule algebra $B$, the symbol $\qMod{A}{B}{}{B}$ is reserved for the
category of relative Hopf modules.
Thus, objects of $\qMod{A}{B}{}{B}$ are pairs $(M, \prescript{M}{}{\delta})$ where $M$ is a~$B$-bimodule and
$\prescript{M}{}{\delta}: M \to A\otimes M $ is a~coaction such that  
\[
  \prescript{M}{}{\delta(bmc)}= \sum b_{(1)}m_{(-1)} c_{(1)}\otimes b_{(2)}m_{(0)}c_{(2)}\quad \text{for all $b,c \in B$, $m \in M$}.
\]
A morphism $f: M \to N$ in $\qMod{A}{B}{}{B}$ is a~left $A$ colinear map which is also $B$-bilinear.

\begin{defn} \label{ Quantum homogeneous space }
Let $A, H$ be Hopf-algebras and $\pi:A\to H$ be a surjective Hopf-algebra homomorphism. Consider the homogeneous right $H$-coaction on $A$  given by $\delta^A =(\id\otimes \pi)\circ \Delta$ and let $B= A^{\co(H)} := \{ a\in A: \delta^A (a) = a\otimes 1\}$
be the coinvariant subalgebra.
We say that $B$ is a quantum homogeneous space if $A$ is faithfully flat as a right $B$-module.
\end{defn}

It is easy to see that if $B$ is a quantum homogeneous space of a Hopf-algebra $(A, \Delta)$, then the pair $(B, \Delta)$ is an object of the category $ \qMod{A}{B}{}{B}$.

  Now we recall the following well-known result.

\begin{prop} \label{12thdec236}
Let $B$ be a left $A$-comodule algebra and $M $ be an object of the monoidal category $ \qMod{A}{B}{}{B}$. If $M$ is  finitely generated and projective as a left $B$ module, consider the map 

  $$  \prescript{\prescript{}{B}{\Hom} (M, B)}{}{\delta}: \prescript{}{B}{\Hom(M,B)} \to A \otimes \prescript{}{B}{\Hom(M,B)} $$
  defined by $ \prescript{\prescript{}{B}{\Hom} (M, B)}{}{\delta}(f) =  f_{(-1)}\otimes f_{(0)}$, where
\begin{align}\label{27thnov233}
   \prescript{\prescript{}{B}{\Hom} (M, B)}{}{\delta}(f)  (m)= f_{(-1)}\otimes f_{(0)}(m) = S(m_{(-1)})[f(m_{(0)})]_{(-1)}\otimes [f(m_{0})]_{(0)}
\end{align}
   for $m\in M$. 
Then $ (\prescript{}{B}{\Hom} (M, B), \prescript{\prescript{}{B}{\Hom} (M, B)}{}{\delta})  $ is a right dual of $M$ in the category
$ \qMod{A}{B}{}{B}$.
\end{prop}
\begin{proof}  \cite[Lemma 2.2]{Ulb90} implies  that $ (\prescript{}{B}{\Hom} (M, B), \prescript{\prescript{}{B}{\Hom} (M, B)}{}{\delta})   $  is a comodule while \cite[Proposition A.9]{OSV} proves that $\prescript{}{B}{\Hom} (M, B)$ is an object in $\qMod{A}{B}{}{B}$. It is well known that the maps $\ev, \coev$ defined in \eqref{evcoev} are morphisms in $\qMod{}{B}{}{B}$. By a routine computation, it can be checked that $\ev $ and $ \coev$ are also left $A$-colinear.
\end{proof}

\subsection{Takeuchi's equivalence} \label{19thdec235}

Throughout this subsection, $ A$ will denote a Hopf-algebra and $B = A^{\co  H}$ a quantum homogeneous space as in
Definition \ref{ Quantum homogeneous space }. We will need a few definitions and notations before stating the form of Takeuchi's equivalence which will be needed repeatedly in this article. 

Firstly, $\qMod{A}{B}{}{B}$  will continue to denote the category of relative Hopf-modules introduced in Subsection \ref{6thjune241}. Moreover, $ \qMod{H}{}{}{B} $ will denote the category whose objects are left $H$-comodules $ (V, \prescript{V}{}{\delta}) $ such that $V$ is a right $B$-module and 
$$ \prescript{V}{}{\delta} (v b) = v_{(- 1)} \pi (b_{(1)}) \otimes v_{(0)} b_{(2)} $$
for all $v \in V$ and $b \in B$. Here, $\pi: A \rightarrow H$ is the surjective Hopf-algebra homomorphism as in Definition \ref{ Quantum homogeneous space }. 

Morphisms in $ \qMod{H}{}{}{B} $ are left $H$-comodule maps which are also right $B$-linear.

Let $B^+$ denote $B \cap \ker (\epsilon). $ If $M$ is an object of $\qMod{A}{B}{}{B}$, then the map
$$  \frac{M}{B^+ M} \rightarrow H \otimes  \frac{M}{B^+ M}, ~  [ m ]  \mapsto \pi (m_{(- 1)}) \otimes [ m_{(0)} ] $$
is a coaction of $H$ on $ \frac{M}{B^+ M}$,
where $[ m ]$ denotes the equivalence class of $m \in M$ in $ \frac{M}{B^+ M}$. Thus, $ \frac{M}{B^+ M}$ is an object of the category  $\qMod{H}{}{}{B}$ with respect to the obvious right $B$-module structure  and we have a functor
\begin{equation} \label{2ndjuly241}
 \Phi: \qMod{A}{B}{}{B} \rightarrow \qMod{H}{}{}{B}, ~ M \rightarrow  \frac{M}{B^+ M}.
 \end{equation}
In the other direction, we have a functor 
\begin{equation} \label{2ndjuly242}
 \Psi: \qMod{H}{}{}{B}  \rightarrow \qMod{A}{B}{}{B}, ~ \Psi (V) = A \square_H V,
\end{equation}
where $A \square_H V$ is the cotensor product of $A$ and $V$ defined  to be the kernel of the map 
$ \delta^{A} \otimes \id - \id \otimes \prescript{V}{}{\delta}: A \otimes V \rightarrow A \otimes H \otimes V. $

 The left $A$-comodule structure of $\Psi (V)$  is defined on the first tensor factor, the right B-module structure is the diagonal one, and if  $\gamma$ is a morphism in $  \qMod{H}{}{}{B}, $ then $ \Psi (\gamma):= \id \otimes \gamma. $

\begin{defn} \label{10thjune241}
We define $ \modz{A}{B} $ to be the full subcategory of $ \qMod{A}{B}{}{B} $  whose objects $M$  are finitely generated as left $B$-modules and satisfy the condition $M B^+ = B^+ M$. Moreover, $\lmod{H}{}$ will denote the category of finite dimensional left $H$-comodules.
\end{defn}

The categories $\modz{A}{B}$ and $\lmod{H}{}$ are monoidal, the monoidal structure on $\modz{A}{B}$ is the usual tensor product of comodules and  $B$-bimodules while the monoidal structure on $\lmod{H}{}$  is given by the usual tensor product of comodules.

Then we have the following result:

\begin{thm}     \label{3rdapril241}
If $B = A^{\co(H)} $ is a quantum homogeneous space, then the functor
\[
  \Phi: \modz{A}{B} \rightarrow \lmod{H}{} 
\]
is a monoidal equivalence of categories.

In particular, if $M$ is a monoid object in $\modz{A}{B}$, then $\Phi (M) $ is a monoid object in $\lmod{H}{}$.
\end{thm}

We refer to Section 4 of \cite{Tak,Skryabin2007} for the details.

\begin{rem} \label{27thsep241}
We observe that any object of $\modz{A}{B}$ is automatically projective as a left $B$-module.

Indeed, the category $\lmod{H}{} $ is right rigid (see Proposition \ref{12thdec236}) and since $\Phi$ is a categorical equivalence, the category $\modz{A}{B}$ is also right rigid. Thus any object $M$ in $\modz{A}{B}$ admits a right dual in the category $\qMod{}{B}{}{B}$. Therefore, $M$ is projective as a left $B$-module. 
\end{rem}

\subsection{Differential calculi and metrics}

A differential calculus over an algebra $B$ is a differential graded algebra $(\Omega^{\bullet}(B) = \bigoplus_{k\ge 0} \Omega^k(B),  \wedge, d)$ such that $\Omega^0(B) = B$ and $ \Omega^\bullet (B) $ is generated as an algebra by $B$ and $dB$. 

 A differential calculus $(\Omega^{\bullet}(B), \wedge, d) $ on a $\ast$-algebra  $B$ is called a $\ast$-differential calculus if there exists a conjugate linear involution $\ast: \Omega^{\bullet}(B) \to \Omega^{\bullet}(B) $ which extends the map $\ast: B \rightarrow B $ such that 
  $$ \ast(\Omega^k(B))\subseteq \Omega^k(B), ~ (d\omega)^{\ast} =d(\omega^{\ast}), ~ (\omega \wedge \nu)^{\ast} = (-1)^{kl} \nu^{\ast} \wedge \omega^{\ast} $$
for  all $ \omega \in \Omega^{k}(B), \nu \in \Omega^l (B)$.

 The main examples of differential calculi in this article come from certain comodule algebras (more precisely, quantum homogeneous spaces in the sense of Definition \ref{ Quantum homogeneous space }) where the differential calculus is compatible with the Hopf-algebra coaction. 

\begin{defn} \label{Covariant Calculi}
Let $(B, \prescript{B}{}{\delta})$ be a left $A$-comodule algebra for a Hopf-algebra $A$. A differential calculus $(\Omega^{\bullet} (B), \wedge, d)$ on $B$ is called left $A$-covariant if the coaction $ \prescript{B}{}{\delta} : B \rightarrow A \otimes B$ extends to a (necessarily unique) comodule algebra map $\prescript{\Omega^\bullet}{}{\delta}: \Omega^\bullet (B) \rightarrow A \otimes \Omega^\bullet (B) $ such that  $\Omega^k$ is an object of $\qMod{A}{B}{}{B}$ for each $k \geq 0$ and the map $d$ is $A$-covariant.

If $A$ is a Hopf $\ast$-algebra and $B$ is a comodule $\ast$-algebra, then a covariant differential calculus  $(\Omega^{\bullet} (B), \wedge, d)$ is called a $\ast$-differential calculus if  we additionally demand that $\Omega^\bullet (B)$ is a comodule $\ast$-algebra.
\end{defn}

If $(\Omega^{\bullet} (B), \wedge, d)$ is a left $A$-covariant calculus, then clearly, the canonical projection map $\Omega^\bullet (B) \rightarrow \Omega^k (B) $ as well as $\wedge: \Omega^k (B) \otimes_B \Omega^l (B) \rightarrow \Omega^{k + l} (B) $ are left $A$-covariant.

From now on, while referring to a differential calculus on $B$, we will often use the notations $\Omega^k$ and $\Omega^{\bullet}$ to denote $\Omega^k(B)$ and $\Omega^{\bullet}(B)$ respectively.

\begin{defn} (\cite{BeggsMajid:Leabh}) \label{4thmay242}
A metric on a differential calculus $(\Omega^\bullet, \wedge, d)$ on an algebra $B$ is a pair $(g, (~ , ~))$ where $g $ is an element of $\Omega^1 \otimes_B \Omega^1$ and $(~ , ~): \Omega^1 \otimes_B \Omega^1 \rightarrow B$ is a $B$-bilinear map such that the following conditions hold:
$$ ((\omega, ~) \otimes_B \id) g = \omega = (\id \otimes_B (~ , \omega)) g.  $$
A metric is said to be quantum symmetric if $\wedge g = 0$.
If the differential calculus is $A$-covariant, then we will say that the metric is covariant if  $(~ , ~)$ and the map 
\begin{equation} \label{23rdmay241}
\coev_g: B \rightarrow \Omega^1 \otimes_B \Omega^1 ~  \text{defined by} ~ \coev_g (b) = b g 
\end{equation}
are $A$-covariant.
\end{defn}

If $ (g, (~ , ~)) $ is a metric on $B$, then  by \cite[Lemma 1.16]{BeggsMajid:Leabh}, $g$ is central, i.e, $b g = g b$ for all $b \in B$. This implies that the map $\coev_g$ is actually $B$-bilinear. The following well-known characterization (see page 311 of \cite{BeggsMajid:Leabh}) of metrics will be used throughout the article repeatedly.

\begin{remark} \label{24thjuly241}
If $ (g, (~ , ~)) $ is a covariant metric on $\Omega^1$ and $\coev_g$ is defined in \eqref{23rdmay241}, then $ (\Omega^1, (~ , ~), \coev_g) $ is a right dual of $\Omega^1$ in the category $\qMod{A}{B}{}{B}$. 

Conversely, if a triplet $ (\Omega^1, \ev, \coev) $ is a right dual of $\Omega^1$ in $\qMod{A}{B}{}{B}$, then the pair $ (\coev (1), \ev) $  is a covariant metric on $\Omega^1$.
\end{remark}
Let us also make a note of the following well-known fact.
\begin{lem} \label{8thmay24jb1}
If $ (x, (~ ,~)) $ and $ (y, (~,~)) $ are two metrics on $\Omega^1, $ then $x = y$. Conversely, if $(x,(~,~)_1)$ and $(x, (~,~)_2)$ are two metrics on $\Omega^1$, then $(~,~)_1=(~,~)_2$.
\end{lem}

As in \cite{BeggsMajid:Leabh}, we consider the well-defined antilinear map 
$$\dagger = \mathrm{flip}(*\otimes *) : \Omega^1(B) \otimes_B \Omega^1(B) \to \Omega^1(B) \otimes_B \Omega^1(B).$$

\begin{defn} \label{3rdjuly244}
 We say that a metric $(g, (~,~)) $ is real of $g^{\dagger} = g$.    
\end{defn}

Then we have the following corollary to Lemma \ref{8thmay24jb1}.

\begin{cor} \label{9thmay24jb2}
    If $(\Omega^\bullet, \wedge, d) $ is a $*$-differential calculus on $*$-algebra $B$ and  $(g, (~,~))$ is metric on $\Omega^1$ then $g^\dagger = g $ if and only if $(\omega, \eta)= (\eta^*, \omega^*)^*$ for all $\omega, \eta \in \Omega^1$.
\end{cor}
\begin{proof}
    Suppose $(g, (~,~))$ is a metric on  $\Omega^1(B)$. Then it is easy to check that $(g^\dagger, (~,~)_1)$ is also a metric where $(\omega,\eta)_1= (\eta^*, \omega^*)^*$. So, by  Lemma \ref{8thmay24jb1}, $g=g^\dagger $ if and only if $(\omega, \eta)= (\eta^*, \omega^*)^*$ for all $\omega, \eta \in \Omega^1$.
\end{proof}

\subsection{Complex structures}

Following the work of Polishchuk and Schwarz in \cite{PolSchwar2003}, the notion of complex structures for noncommutative differential calculus have been studied by many mathematicians, see for example, \cite{KLvSPodles}, \cite{BS}, and references therein. We will work with the following definition introduced in \cite{MMF2} where $\mathbb{N}_0$ will denote the set $\mathbb{N} \cup \{ 0 \}. $

\begin{defn} (\cite{MMF2})
  A {\em  complex structure} $\Om^{(\bullet,\bullet)}$ for a~differential $*$-calculus~$(\Om^{\bullet}, \wedge, d)$ on a $\ast$-algebra $B$
  is an $\bN^2_0$-algebra grading $\bigoplus_{(p,q)\in \bN^2_0} \Om^{(p,q)}$
  for $\Om^{\bullet}$ such that, for all $(p,q) \in \bN^2_0$, we have 
  \[
    \Om^k = \bigoplus_{p+q = k} \Om^{(p,q)},\qquad
    \big(\Om^{(p,q)}\big)^* = \Om^{(q,p)},\qquad
    \exd \Om^{(p,q)} \sseq \Om^{(p+1,q)} \oplus \Om^{(p,q+1)}.
  \]
\end{defn}
Here, $\Omega^{(0,0)} = \Omega^0 = B$. As $ \Omega^\bullet = \bigoplus_{(p,q)\in \bN^2_0} \Om^{(p,q)}$ is an $\bN^2_0$-algebra grading, it follows that $\Omega^{(p,q)}$ is a $B$-bimodule.

A combination of \cite[Lemma 2.15 and Remark 2.16]{MMF2}    implies that the above definition is equivalent to the definition of complex structures given by Beggs and Smith in \cite{BS} as well as that by Khalkhali, Landi and van Suijlekom in \cite{KLvSPodles}.

If $\Om^{(\bullet,\bullet)}$ is a complex structure, then for $p, q \in \bN_0$,  $\pi^{p,q}: \Omega^{p + q} \rightarrow \Omega^{(p, q)} $ will denote the canonical projections associated to the decomposition $  \Om^k = \bigoplus_{p+q = k} \Om^{(p,q)}$. Moreover, $\partial$ and $\overline{\partial}$ will denote the maps
$$ \partial:= \pi^{(p + 1, q)} \circ d: \Omega^{(p, q)} \rightarrow \Omega^{(p + 1, q)} ~ \text{and} ~ \overline{\partial}:= \pi^{(p, q + 1)} \circ d: \Omega^{(p,q)} \rightarrow \Omega^{(p, q + 1)}. $$
Then, from \cite{BS}, we have the following results, which will be used repeatedly in this article. 
\begin{lem}  \cite[ Section 3]{BS}  \label{12thdec233}
For a complex structure $\Omega^{(\bullet, \bullet)}$ on a $\ast$-algebra $B$, the following statements hold:
\begin{enumerate}
    \item  $d = \partial + \overline{\partial}. $ 
    \item The maps $\partial$ and $\overline{\partial}$ satisfy graded-Leibniz rule.
    \item For all $\omega \in \Omega^\bullet$, we have
    $$ \partial (\omega^\ast) = (\overline{\partial} (\omega))^\ast, ~ \overline{\partial} (\omega^\ast) = (\partial (\omega))^\ast. $$
\end{enumerate}
\end{lem}

In the sequel, we will denote a complex structure on a differential calculus $(\Omega^\bullet, \wedge, d)$ by the quadruple $(\Omega^{(\bullet, \bullet)}, \wedge, \partial, \overline{\partial})$.

\begin{defn} \label{12thdec234}
A complex structure $ (\Omega^{(\bullet, \bullet)}, \wedge, \partial, \overline{\partial}) $ on a left $A$-comodule $\ast$-algebra $B$ is said to be left $A$-covariant if  $ (\Omega^\bullet, \wedge, d) $ is a covariant differential $\ast$-calculus in the sense of Definition \ref{Covariant Calculi} and moreover, if $\Omega^{(p,q)}$ is a left $A$-comodule for all $ (p, q) \in \bN^2_0. $
\end{defn}

In this case, the covariance of  $(\Omega^\bullet, \wedge, d)$ implies that the projections $\pi^{p,q}$ and $d$ are $A$-covariant and so the maps $\partial$ and $\overline{\partial}$ are also $A$-covariant.

\subsection{Bar Categories}
  Following \cite{BarCategory}, we recall the definition and some examples of bar categories which will be relevant for us. For a monoidal category $ (\mathcal{C}, \otimes)$ and objects $X, Y $ of $ \mathcal{C}, $ we denote by flip the functor from  $\mathcal{C}\times \mathcal{C} $ to $\mathcal{C}  \times \mathcal{C}$ which sends the pair $(X,Y)$ to $ (Y,X)$. We also denote by $1_{\mathcal{C}}$ the unit object. As usual, we will suppress the notations for the left unit, right unit as well as the associator of $\mathcal{C}$.
  
  \begin{defn} \label{15thjuly241}
  A bar category is a monoidal category $(\mathcal{C},\otimes,1_{\mathcal{C}})$ together with  a functor $\mathrm{bar}: \mathcal{C}\to \mathcal{C}$ (written as $X\mapsto \overline{X}$),
      a natural equivalence $\mathrm{bb}: \id_ {\mathcal{C}} \to \mathrm{bar}\circ \mathrm{bar}  $ between the identity and the bar $\circ$ bar functors on $\mathcal{C}$,
      an invertible morphism $ \star:1_{\mathcal{C}} \to \overline{1_{\mathcal{C}}}$ and
      a natural equivalence $\Upsilon$  between $\mathrm{bar}\circ \otimes$ and $\otimes \circ (\mathrm{bar}\times \mathrm{bar})\circ \mathrm{flip}$ from $\mathcal{C}\times \mathcal{C}$ to $\mathcal{C}$
  
such that the following compositions of morphisms are both equal to $1_{\overline{X}}:$
   $$ \overline{X} \xrightarrow{\cong} \overline{X \otimes 1_{\mathcal{C}}} \xrightarrow{\Upsilon_{X, 1_{\mathcal{C}}}} \overline{ 1_{\mathcal{C}}}\otimes  \overline{X} \xrightarrow{\star^{-1} \otimes \id}  1_{\mathcal{C}} \otimes \overline{X} \xrightarrow{\cong} \overline{X}, $$
   $$ \overline{X} \xrightarrow{\cong} \overline{ 1_{\mathcal{C}} \otimes X} \xrightarrow{\Upsilon_{ 1_{\mathcal{C}},X}}  \overline{X} \otimes\overline{1_{\mathcal{C}}}\xrightarrow{ \id \otimes  \star^{-1}}  \overline{X} \otimes  1_{\mathcal{C}} \xrightarrow{\cong} \overline{X},$$ 
  and moreover, the following equations hold:
$$ (\Upsilon_{Y,Z} \otimes \id) \Upsilon_{X, Y \otimes Z} = (\id \otimes \Upsilon_{X, Y}) \Upsilon_{X \otimes Y, Z}, ~ \overline{\star} \star = \mathrm{bb}_{1_{\mathcal{C}}}: 1_{\mathcal{C}} \to \overline{\overline{1_{\mathcal{C}}}}, ~ \overline{\mathrm{bb}_X} =\mathrm{bb}_{\ol{X}}: \overline{X} \to \overline{\overline{\overline{X}}}$$
for all objects $X, Y, Z$ in $\mathcal{C}$.

An object $X$ in a bar category is  called a star object if there is a morphism $\star_X: X\to \overline{X}$ such that $\ol{\star_X}\circ \star_X = \mathrm{bb}_X $.
 \end{defn}

In this article, for a vector space $M$, the symbol $\overline{M}$ will denote the vector space defined as 
$$ \overline{M}:= \{ \overline{m}: m \in M \}.$$ 
Moreover, if $B$ is a~$\ast$-algebra and $M$ a~$B$-bimodule, then $\overline{M}$ is equipped with the following $B$-bimodule structure:
\begin{align}\label{28thnov231}
  b\cdot \ol{m}= \ol{m\cdot b^{\ast}};\quad \ol{m}\cdot b= \ol{b^{\ast}\cdot m} \quad \text{for all $b\in B$, $m\in M$}.
\end{align}
If $(M, \prescript{M}{}{\delta})$ is a~left $A$-comodule for a~Hopf $\ast$-algebra $A$, then $\overline{M}$ has a~left $A$-comodule structure defined by
\begin{align}\label{25thnov232}
  \prescript{\ol{M}}{}{\delta}(\ol{m})= m_{(-1)}^{\ast}\otimes \ol{m_{(0)}},
\end{align}
where we have used Sweedler's notation $\prescript{M}{}{\delta}(m)= m_{(-1)}\otimes m_{(0)}$.

We will need the following examples of bar categories in this article. 

\begin{eg} (\cite[Section 2.8]{BeggsMajid:Leabh}) \label{2nddec231}

\begin{enumerate}
    \item If $B$ is a $\ast$-algebra, then $\qMod{}{B}{}{B}$ is a bar category.
      If $M$ is an object of $\qMod{}{B}{}{B}$, then  $\mathrm{bar} (M) := \overline{M}$,
      where $\overline{M}$ is the $B$-bimodule defined in~\eqref{28thnov231}.
      For $f \in \Hom(M,N)$, we define 
      $ \ol{f}\in \Hom(\ol{M}, \ol{N})$ by $\ol{f}(\ol{x})=\ol{f(x)}$.
      For more details, we refer to~\cite[Section 2.8]{BeggsMajid:Leabh}.

    \item If $A$ is a Hopf $\ast$-algebra, then $\lMod{A}{}$ is a bar category where $\mathrm{bar} (M) = \overline{M} $ and the left $A$-coaction on $\overline{M}$ is defined  by \eqref{25thnov232}.

    \item If A is a Hopf $\ast$-algebra and $B$ a left $A$-comodule $\ast$-algebra, then the category $\qMod{A}{B}{}{B}$ of relative Hopf modules is a bar category.

	Finally, let us note that if $(\Omega^{\bullet},\wedge, d)$ is an $A$-covariant $\ast$-differential calculus on $B$, then the map
 \begin{equation} \label{8thjuly241}
  \star_{\Omega^1}: \Omega^1 \rightarrow \overline{\Omega^1}, ~ \star_{\Omega^1} (\omega) := \overline{\omega^*} 
  \end{equation}
makes  $\Omega^1 $  a star object of $\qMod{A}{B}{}{B}$.
\end{enumerate}
\end{eg}

\begin{rem} \label{11thdec23n1}
Let $ (\Omega^{(\bullet, \bullet)}, \wedge, \partial, \overline{\partial}) $ be a complex structure over a $\ast$-algebra $B$. Then, as explained above, $\Omega^1$ is a star-object of the bar-category $\qMod{}{B}{}{B}$. Moreover, Beggs and Smith observed in \cite{BS} that the restrictions of the map $\star_{\Omega^1}$ on $\Omega^{(1,0)}$ and $\Omega^{(0,1)}$ are isomorphisms onto $\overline{\Omega^{(0,1)}}$ and $\overline{\Omega^{(1,0)}}$ respectively. 
Indeed, this follows from the fact that $ (\Omega^{(p,q)})^\ast = \Omega^{(q,p)}. $
\end{rem}

The following result will play a key role in this article.

\begin{prop} \label{10thmay24jb1}
If $B = A^{\co(H)}$ is a quantum homogeneous space of a Hopf $\ast$-algebra $A$, then the category $\modz{A}{B}$ introduced in Definition \ref{10thjune241} is a bar category. 
\end{prop}
\begin{proof}
This has been proved in the appendix. See Proposition \ref{3rdjuly243}. 
\end{proof}

\subsection{Connections} \label{4thaugust241}

If  $(\Omega^{\bullet}, \wedge, d) $ is a differential calculus over an algebra $B$, then a left connection on a $B$-bimodule $\mathcal{E}$ is a $\mathbb{C}$-linear map
$\nabla: \mathcal{E} \rightarrow \Omega^1 \otimes_B \mathcal{E}$ such that
$\nabla (b e) = b \nabla (e)  + db \otimes_B e $
for all $e \in \mathcal{E}$ and for all $b \in B$.

A left connection $\nabla$ on a $B$-bimodule $\mathcal{E}$ is called a left $\sigma$-bimodule connection if there exists a $B$-bimodule map 
$\sigma: \mathcal{E} \otimes_B \Omega^1 \rightarrow \Omega^1 \otimes_B \mathcal{E}$
such that
\begin{equation} \label{19thoct237}
\nabla (e b) = \nabla (e) b + \sigma (e \otimes_B db)
\end{equation}
for all $e \in \mathcal{E}$ and for all $b \in B$.

If \eqref{19thoct237} is satisfied, then we will sometimes say that $ (\nabla, \sigma) $ is a left bimodule connection.

Suppose $(\nabla_{\mathcal{E}}, \sigma_{\mathcal{E}}) $ and $ (\nabla_ {\mathcal{F}}, \sigma_ {\mathcal{F}}) $ are bimodule connections  on $B$-bimodules $\mathcal{E}$ and $\mathcal{F}$.  Then by \cite[Theorem 3.78]{BeggsMajid:Leabh},  we have a bimodule connection $ (\nabla_{\mathcal{E} \otimes_B \mathcal{F}}, \sigma_{\mathcal{E} \otimes_B \mathcal{F}}) $ on $\mathcal{E} \otimes_B \mathcal{F} $ defined as
\begin{equation} \label{4thmay241}
 \nabla_{\mathcal{E} \otimes_B \mathcal{F}}:= \nabla_{\mathcal{E}} \otimes_B \id + (\sigma_{\mathcal{E}} \otimes_B \id) (\id \otimes_B \nabla_{\mathcal{F}}), ~ \sigma_ {\mathcal{E} \otimes_B \mathcal{F}}  =  (\sigma_{\mathcal{E}} \otimes_B \id)  (\id \otimes_B \sigma_{\mathcal{F}}).
\end{equation}
Thus, the following definition makes sense.

\begin{defn} (\cite{BeggsMajid:Leabh}) \label{4thjuly241}
Suppose $(g, ( ~, ~ )) $ is a metric on the space of one-forms $\Omega^1$ of a differential calculus on an algebra $B$ in the sense of Definition \ref{4thmay242}. Then a bimodule connection  $ (\nabla, \sigma) $ is said to be compatible with $(g, (~,~))$ if $\nabla g = 0$, where, by an abuse of notation, we have denoted the connection on $\Omega^1 \otimes_B \Omega^1$ defined by \eqref{4thmay241} by $\nabla$ again.
\end{defn}

Now we discuss the notions of $\partial$ and $\overline{\partial}$-connections for a complex structure. 

Suppose a $\ast$-algebra $B$ is equipped with a complex structure  $ (\Omega^{(\bullet, \bullet)}, \wedge, \partial, \overline{\partial}) $ and $\mathcal{E}$ is a $B$-bimodule. Then a left  $\overline{\partial}$-connection on $\mathcal{E}$ is a $\mathbb{C}$-linear map $\overline{\partial}_{\mathcal{E}}:\mathcal{E}\to \Omega^{(0,1)}\otimes_B \mathcal{E}$ such that 
   $$ \overline{\partial}_{\mathcal{E}}(ae)= a \overline{\partial}_{\mathcal{E}}(e)+ \overline{\partial}(a)\otimes_B e. $$
Similarly,  a left $\partial$-connection on $\mathcal{E}$ is a $\mathbb{C}$-linear map $\partial_{\mathcal{E}}:\mathcal{E}\to \Omega^{(1,0)}\otimes_B \mathcal{E}$ satisfying the equation
${\partial}_{\mathcal{E}}(ae)= a {\partial}_{\mathcal{E}}(e)+ {\partial}(a)\otimes_B e$.

 Moreover, right $\partial$-connections and right $\overline{\partial}$-connections are defined analogously.

Generalizing the definition of a bimodule connection, Beggs and Majid (\cite{BeggsMajidChern}) called a left $\overline{\partial}$-connection $\overline{\partial}_{\mathcal{E}}$ on $\mathcal{E}$ to be a left $\sigma$-bimodule $\overline{\partial}$-connection if there exists a $B$-bimodule map 
$\sigma: \mathcal{E} \otimes_B \Omega^{(0,1)} \rightarrow \Omega^{(0, 1)} \otimes_B \mathcal{E}$
such that
$$\overline{\partial}_{\mathcal{E}} (e b) = \overline{\partial}_{\mathcal{E}} (e) b + \sigma (e \otimes_B \overline{\partial} b)$$
for all $e \in \mathcal{E}$ and for all $b \in B$. Similarly, we have the notions of right bimodule $\overline{\partial}$-connections and  left (or right) bimodule $\partial$-connections.

We end this subsection with the definition of a holomorphic bimodule. 

\begin{defn} \label{11thdec231}
Suppose $ (\Omega^{(\bullet, \bullet)}, \wedge, \partial, \overline{\partial}) $ is a complex structure on a $\ast$-algebra $B$. The holomorphic curvature of a left $\overline{\partial}$-connection $\overline{\partial}_{\mathcal{E}}$  on a $B$-bimodule $\mathcal{E}$  is the map $R_{\mathcal{E}}^{\Hol}: \mathcal{E} \rightarrow \Omega^{(0, 2)} \otimes_B \mathcal{E} $ defined as:
\begin{align*}
    R_{\mathcal{E}}^{\Hol}(e)=(\ol{\partial}\otimes_B id- id\wedge \ol{\partial}_{\mathcal{E}})\ol{\partial}_{\mathcal{E}} (e).
\end{align*}
A holomorphic structure on $\mathcal{E}$ is the choice of a $\overline{\partial}$-connection $\overline{\partial}_{\mathcal{E}}$ on $\mathcal{E}$ whose holomorphic curvature vanishes. In this case, we say that the pair $(\mathcal{E}, \overline{\partial}_{\mathcal{E}})$ is a holomorphic $B$-bimodule. 

If the complex structure on $B$ is left $A$-covariant in the sense of Definition \ref{12thdec234}, then a holomorphic $B$-bimodule $ (\mathcal{E}, \overline{\partial}_{\mathcal{E}})  $ is called left $A$-covariant if $\mathcal{E}$ is an object of the category $\qMod{A}{B}{}{B}$ and if the map $\overline{\partial}_\mathcal{E}$ is left $A$-covariant.
\end{defn}

\section{Construction of metrics on quantum homogeneous spaces} \label{10thaugust242}

In this section, we start by classifying covariant metrics on a quantum homogeneous space of the form $B = A^{\co(H)}, $ where $A$ is a Hopf $\ast$-algebra and $H$ is a compact quantum group algebra 
(see Subsection \ref{3rdjuly242}). This classification is proved in Theorem \ref{8thmay24jb5}. We will use this result to classify covariant metrics on the space of one-forms of the Heckenberger--Kolb calculus in Theorem \ref{9thmay24jb5}. In the second subsection, we work in the set up of covariant complex structures on a quantum homogeneous space as above. Here, we construct a family of covariant metrics $(g, (~ , ~)) $ which are real in the sense of Definition \ref{3rdjuly244} and satisfies the condition $ (\omega, \eta) = 0 $ if $\omega, \eta \in \Omega^{(1,0)}$ or if $\omega, \eta \in \Omega^{(0,1)}$. 

This result will allow us to construct a class of Hermitian metrics on the space of one-forms of a differential calculus equipped with a complex structure (see Theorem \ref{9thmay24jb21}).  

Throughout this section,  $B=A^{\co(H)}$ will denote a quantum homogeneous space of a Hopf $\ast$-algebra $A$, where $H$ is a compact quantum group algebra. Then we have the monoidal category $\modz{A}{B}$ defined in Definition \ref{10thjune241} which is a bar-category by Proposition \ref{10thmay24jb1}. We will repeatedly use Takeuchi's equivalence $\Phi$ (Theorem \ref{3rdapril241}) which is a monoidal equivalence between the categories $\modz{A}{B}$ and $\lmod{H}{}$. Moreover, we will use the fact that the canonical map \eqref{5thjuly241} defines an isomorphism from $\Phi (\overline{M}) $ to  $  \overline{\Phi (M)} $ for any object $M$ in $\modz{A}{B}$ (see Lemma \ref{27thmarch243}).

\subsection{The moduli space of metrics}

We start by showing the existence of a covariant metric on a quantum homogeneous space of the form $B = A^{\co(H)}, $ where $H$ is a compact quantum group algebra. For the definition of a compact quantum group algebra and its properties relevant for this article, we refer to  Subsection \ref{3rdjuly242}.

So  let us assume that  $H$ is a compact quantum group algebra  and moreover, $\Omega^1$ be the space of one-forms of a covariant $\ast$-differential calculus on $B$ such that $\Omega^1$ is an object of the category $\modz{A}{B}$. We will use the notation $V$ to denote $ \Phi(\Omega^1)$.  Since the map $\star_{\Omega^1}$ defined in \eqref{8thjuly241} is an isomorphism in $\modz{A}{B}$, we know that $\Phi(\star_{\Omega^1}) $ induces an isomorphism in $\lmod{H}{}$  from $ V $ to $ \overline{V}$ via Lemma \ref{27thmarch243}.  Moreover, if we fix an $H$-invariant inner product on $V$ (conjugate-linear on the right) in the sense of \eqref{28thmarch24n2}, then we have the isomorphism $\psi: \overline{V}\to \Hom(V, \mathbb{C})$ in $\lmod{H}{}$  from Lemma \ref{28feb241}. Hence,  we have a right dual pairing of $V$ with itself defined by the following  morphisms in $\lmod{H}{}$: 
\begin{align}
    V\otimes V \xrightarrow{\id\otimes \Phi(\star_{\Omega^1})} V\otimes \overline{V} &\xrightarrow{\id \otimes \psi} V \otimes \Hom(V, \mathbb{C}) \xrightarrow{\ev_{V} }\mathbb{C};
    \label{22ndapril241}\\
    \mathbb{C}\xrightarrow[]{\coev_V}\Hom(V, \mathbb{C}) \otimes V &\xrightarrow[]{\psi^{-1}\otimes \id}\overline{V}\otimes V \xrightarrow[]{\Phi(\star_{\Omega^1}^{-1})\otimes\id} V \otimes V \label{22ndapril242}.
\end{align}
In the above two equations, we have implicitly used the isomorphism $ \Phi (\overline{\Omega^1}) \rightarrow \overline{\Phi (\Omega^1)} $ from Lemma \ref{27thmarch243}. 

Now, by Takeuchi's equivalence (Theorem \ref{3rdapril241}), we have morphisms $\ev: \Omega^1 \otimes_B \Omega^1 \to B$ and $\coev: B \to \Omega^1 \otimes_B \Omega^1$ in $\modz{A}{B}$ where $\Phi (\ev) $ and $\Phi (\coev) $ are defined by  \eqref{22ndapril241} and \eqref{22ndapril242} respectively and moreover, the pair $(\ev, \coev )$ satisfies \eqref{27thnov231}. Thus, by Remark \ref{24thjuly241}, we immediately have the following result.

\begin{prop} \label{22ndmay241}
 Suppose $ (\Omega^\bullet, \wedge, d) $ is a left $A$-covariant $\ast$-differential calculus on a quantum homogeneous space  $B= A^{\co(H)}$  such that $H$ is a compact quantum group algebra. If $\Omega^1$ is an object of the category $\modz{A}{B}$ and $\langle~,~\rangle$ be an invariant inner product on $V:= \Phi(\Omega^1)$ in the sense of \eqref{28thmarch24n2},  consider the morphisms $\ev: \Omega^1 \otimes_B \Omega^1 \to B$ and $\coev: B \to \Omega^1 \otimes\Omega^1$ as defined above.   If $g=\coev(1_B)$, then the pair $(g, \ev)$ defines an $A$-covariant metric on $B$. 
   \end{prop}

Now, an application of the categorical result of Lemma \ref{22ndapril243} yields the following theorem: 

\begin{thm} \label{8thmay24jb5}
Let $B = A^{\co(H)}$ be a quantum homogeneous space of a Hopf $\ast$-algebra $A$ such that $H$ is a compact quantum group algebra. If $ (\Omega^\bullet, \wedge, d) $ is an $A$-covariant $\ast$-differential calculus on $B$ and $\Omega^1  $ is an object of the category $\modz{A}{B}$, then the space of left $A$-covariant metrics on $\Omega^1  $ admits a free and transitive action of the group $ \Aut (\Phi (\Omega^1)). $

In particular, if $\Phi(\Omega^1) \simeq \bigoplus_{i=1}^m V_i^{\oplus n_i}$
is the decomposition of $\Phi(\Omega^1)$ into irreducible $H$-comodules such that $V_i$ and $V_j$ are not isomorphic for $i \neq j$, then
\begin{align} \label{eqn:auto.group}
 \Aut (\Phi (\Omega^1)) \cong  \prod_{i=1}^m \mathrm{GL}_{n_i}(\mathbb{C}).
\end{align}
In this case, the set of all left $A$-covariant metrics on $\Omega^1$ is parametrized by the group $ \prod_{i=1}^m \mathrm{GL}_{n_i}(\mathbb{C})$.
\end{thm}
\begin{proof}
Since $H$ is a compact quantum group algebra, the set of all $A$-covariant metrics on $\Omega^1  $ is non-empty by Proposition \ref{22ndmay241}. Consider the monoidal category $\modz{A}{B}$. Then by a combination of Remark \ref{24thjuly241} and Lemma \ref{22ndapril243}, it follows  that the space of $A$-covariant  metrics is acted upon freely and transitively by the automorphism group of $\Omega^1 $ in the category $\modz{A}{B}$. However, by Takeuchi's equivalence, this group is isomorphic to the automorphism group of $\Phi(\Omega^1)$ in the category $\lmod{H}{}$.

The proof of \eqref{eqn:auto.group} now follows from Schur's lemma and the fact that $ \End (V^{\oplus n_i}_i) \cong M_{n_i} (\mathbb{C}) $  (see the proof of Proposition 2.2.15 in \cite{NeshveyevTuset}).
\end{proof}

\subsection{Metrics on quantum homogeneous spaces with a complex structure}

Throughout this subsection,  $(\Omega^{(\bullet,\bullet)}, \wedge, \partial, \overline{\partial})$ will denote an $A$-covariant complex structure on a quantum homogeneous space $B=A^{\co(H)}$ such that $\Omega^1$ is an object of the category $\modz{A}{B}$. It can be easily seen that under this assumption, $\Omega^{(1,0)}$ and $\Omega^{(0,1)}$ are also objects of $\modz{A}{B}$.  We will use the following notations: 
$$V^{(1,0)}:= \Phi(\Omega^{(1,0)}) = \frac{\Omega^{(1,0)}}{B^+ \Omega^{(1,0)}} \text{ and } V^{(0,1)}:= \Phi(\Omega^{(0,1)}) =  \frac{\Omega^{(0,1)}}{B^+ \Omega^{(0,1)}} .$$
Moreover, for an object $M$ in $\modz{A}{B}$ and $m \in M$, the symbol $[ m ]$ will denote the equivalence class of $m$ in $\frac{M}{B^+ M}$. We also recall from Remark \ref{11thdec23n1} that the map $\star_{\Omega^1}: \Omega^1 \rightarrow \overline{\Omega^1}$ of \eqref{8thjuly241} restricts to an isomorphism from $\Omega^{(0,1)}$ onto $\overline{\Omega^{(1,0)}}$.

\begin{lem} \label{29thmarch24jb3}
      Let $B=A^{\co(H)}$ be a quantum homogeneous space, where $H$ is a compact quantum group algebra and $(\Omega^{(\bullet,\bullet)}, \wedge, \partial, \overline{\partial})$  an $A$-covariant complex structure on $B$. Let $\langle ~, ~\rangle$ be an $H$-invariant inner product on $V^{(1,0)}$. Then $V^{(0,1)}$ is a right dual of $V^{(1,0)}$ in $\lmod{H}{}$ via the morphisms:
    \begin{align}{\label{29thmarch24jb2}}
        &\ev_{V^{(1,0)}}:V^{(1,0)}\otimes V^{(0,1)}\to \mathbb{C}; \quad \ev_{V^{(1,0)}}([\omega] \otimes [\eta])= \langle [\omega] , [\eta^*]\rangle\\
        &\coev_{V^{(1,0)}}:\mathbb{C}\to V^{(0,1)}\otimes V^{(1,0)}; \quad \coev_{V^{(1,0)}}(1)= \sum_{i=1}^{\dim(V)}\Phi(\star_{\Omega^1}^{-1})(\overline{v_i})\otimes v_i,
    \end{align} where $\{v_i\}_{i=1}^{\dim(V)}$ is an  orthonormal basis of $(V^{(1,0)},\langle ~,~\rangle)$ and we have implicitly used the isomorphism    $ \overline{V^{(1,0)}} = \overline{\Phi (\Omega^{(1,0)}) } \rightarrow \Phi (\overline{\Omega^{(1,0)}}) $ via  Lemma \ref{27thmarch243}. 
\end{lem}
\begin{proof}
  Using the fact that $M B^+ = B^+ M$ for any object $M$ in $\modz{A}{B}$ (see Definition \ref{10thjune241}), it can be easily checked that the map $\ev_{V^{(1,0)}}$ is well-defined. Moreover, the definition of $\coev_{V^{(1,0)}}$ is independent of the choice of the orthonormal basis.
  
     Since $\langle~,~\rangle$ is an $H$-invariant inner product on $V^{(1,0)}$, Lemma \ref{28feb241}  implies that $\overline{V^{(1,0)}}$ is a right dual of $V^{(1,0)}$ via the evaluation and coevaluation maps defined in \eqref{30thmarch241}. Moreover, since   $\star_{\Omega^1}: \Omega^{(0,1)} \to \overline{\Omega^{(1,0)}}$  an isomorphism in $\modz{A}{B}, $ then using Lemma \ref{27thmarch243}, it follows that the map $\Phi(\star^{-1}_{\Omega^1}):\overline{V^{(1,0)}} \to V^{(0,1)}$ induces an isomorphism in $\lmod{H}{}$. Now the lemma can be verified by using that  $\Phi$ is monoidal. 
\end{proof}

We continue to use the notations and assumptions of Lemma \ref{29thmarch24jb3} in the next result, the proof of which is immediate.  
\begin{cor} \label{29thmarch24jb4}
    Let $\langle~,~\rangle_{V^{(1,0)}}$ be an $H$-invariant inner product on $V^{(1,0)}$.  We define $$(~,~)^{1,0}:\Omega^{(1,0)} \otimes_B \Omega^{(0,1)}\to B\text{ and }\coev^{1,0}: B \to \Omega^{(0,1)} \otimes_B \Omega^{(1,0)}$$ as $\Phi((~,~)^{1,0})=\ev_{V^{(1,0)}}$ and $\Phi(\coev^{1,0})=\coev_{V^{(1,0)}}$. Then $(\Omega^{(0,1)}, (~,~)^{1,0}, \coev^{1,0})$ is a right dual of $\Omega^{(1,0)}$ in the category $\modz{A}{B}$.
\end{cor}

At this point, let us make an observation. In what follows, we will identify $\Phi (B) $ with $ \mathbb{C}, $ via the isomorphism $[ b ] \mapsto \epsilon (b), $ where   $[ b ]$ denotes the class of $b \in B$ in  $\Phi (B) = \frac{B}{B^+ B}$.  
Now,  by virtue of \eqref{29thmarch24jb2} and the definition of $ (~,~)^{1,0}, $ we have
    \begin{equation}{\label{29thmarch24jb1}}
        [(\omega,\eta)^{1,0}]= \langle [\omega],[\eta^*]\rangle_{V^{(1,0)}} \quad \text{for all $\omega \in \Omega^{(1,0)}$, $\eta\in \Omega^{(0,1)}$}.
    \end{equation}
Therefore, for $\omega \in \Omega^{(1,0)}, \eta\in \Omega^{(0,1)}$, we obtain
    \begin{align*}
        \Phi((~,~)^{1,0})([\omega]\otimes [\eta])&= \langle [\omega], [\eta^*]\rangle_{V^{(1,0)}}\quad \text{(by~\eqref{29thmarch24jb1})}\\
        &=\overline{\langle[\eta^*], [\omega]\rangle_{V^{(1,0)}}}\\
        &= \overline{[(\eta^*,\omega^*)^{1,0}]}\quad \text{(by~\eqref{29thmarch24jb1})}\\
        &= \overline{\epsilon((\eta^*,\omega^*)^{1,0})}\\
        &= \epsilon (((\eta^*,\omega^*)^{1,0})^*).
    \end{align*}
    Thus, 
    \begin{equation} \label{19thapril244}
\Phi((~,~)^{1,0})([\omega]\otimes [\eta]) = \epsilon (((\eta^*,\omega^*)^{1,0})^*).
    \end{equation}

\begin{remark}\label{29thmarch24jb5}
    Let $\langle~,~\rangle_{V^{(0,1)}}$ be an $H$-invariant inner product on $V^{(0,1)}=\Phi(\Omega^{(0,1)})$. A verbatim adaptation of Lemma \ref{29thmarch24jb3} yields morphisms $\ev_{V^{(0,1)}}: V^{(0,1)}\otimes V^{(1,0)}\to \mathbb{C}$ and $\coev_{V^{(0,1)}}: \mathbb{C} \to V^{(1,0)}\otimes V^{(0,1)}$. Thus we have morphisms :
    $(~,~)_{0,1}:\Omega^{(0,1)}\otimes_B \Omega^{(1,0)} \to B$ and $\coev_{0,1}:B\to \Omega^{(1,0)}\otimes_B \Omega^{(0,1)}$ in $\modz{A}{B}$ making $\Omega^{(1,0)}$ a right dual of $\Omega^{(0,1)}$  and we get
    \begin{equation}\label{31stmarch241}
        [(\omega,\eta)^{0,1}]= \langle [\omega],[\eta^*]\rangle_{V^{(0,1)}} \quad \text{for all $\omega \in \Omega^{(0,1)}$, $\eta\in \Omega^{(1,0)}$}.
    \end{equation}
\end{remark}

Now we are in a position to state the main result of this subsection.

\begin{thm} \label{2ndapril242}
    Let $(\Omega^{(\bullet, \bullet)}, \wedge, \partial,\overline{\partial})$ be an $A$-covariant complex structure on a quantum homogeneous space $B=A^{\co(H)}$, where $H$ is a compact quantum group algebra and $\Omega^1$ is an object of the category $\modz{A}{B}$. We fix  $H$-invariant inner products $\langle~,~\rangle_{V^{(1,0)}}$ and $\langle~,~\rangle_{V^{(0,1)}}$ on $V^{(1,0)}$ and $V^{(0,1)}$ respectively.
    Consider the decomposition
    $$ \Omega^1\otimes_B \Omega^1 \cong (\Omega^{(1,0)}\otimes_B \Omega^{(0,1)}) \oplus (\Omega^{(0,1)} \otimes_B \Omega^{(1,0)}) \oplus (\Omega^{(1,0)} \otimes_B \Omega^{(1,0)})\oplus (\Omega^{(0,1)}\otimes_B \Omega^{(0,1)}).$$
       Then for any $\lambda \in \mathbb{R} \setminus \{ 0 \}, $ we  define 
    \begin{align*}
        (~,~)_{\Omega^1}: \Omega^1 \otimes_B \Omega^1 \to B ~ \text{and} ~ 
        \coev_{\Omega^1}: B \to \Omega^1 \otimes_B \Omega^1 
    \end{align*}
    as $(~,~)_{\Omega^1}= (~,~)^{1,0} \oplus (- \lambda^{-1}) (~,~)^{0,1} \oplus 0 \oplus 0$ and $\coev_{\Omega^1}=\coev^{1,0}\oplus (- \lambda) \coev^{0,1}\oplus 0\oplus 0$, where $(~,~)^{1,0}$, $\coev^{1,0}$ and $(~,~)^{0,1},\coev^{0,1}$ are as defined in Corollary \ref{29thmarch24jb4} and Remark \ref{29thmarch24jb5} respectively. If we denote $\coev_{\Omega^1} (1)$ by the symbol $g$, then $(g , (~,~)_{\Omega^1})$ defines an $A$-covariant real metric on $\Omega^1$ (see Definition \ref{3rdjuly244}) which satisfies the  condition
    \begin{equation} \label{18thapril24jb1}
        (\omega, \eta)=0\text{ if }\omega, \eta \in \Omega^{(1,0)} \text{ or if }\omega, \eta \in \Omega^{(0,1)}.
    \end{equation}
    \end{thm}
\begin{proof}
    The fact that $(\coev_{\Omega^1} (1), (~,~)_{\Omega^1})$ is a covariant metric  and the verification of \eqref{18thapril24jb1} follow by  routine computations. For verifying that the metric is real, we introduce a morphism $\xi^{1,0}$ in the category $\modz{A}{B}$. We define  
    \begin{equation*}
        \xi^{1,0}: \Omega^{(1,0)}\otimes_B \Omega^{(0,1)} \to B \quad \text{by} \quad \xi^{1,0} (\omega \otimes_B \eta)=((\eta^*,\omega^*)^{1,0})^*.
    \end{equation*}
Now,
$$ \Phi (\xi^{1,0}) ([ \omega ] \otimes [ \eta ]) = [ {((\eta^*, \omega^*)^{1,0}})^*  ] = \epsilon ({((\eta^*, \omega^*)^{1,0})}^*) = \Phi((~,~)^{1,0})([\omega]\otimes [\eta])  $$
by \eqref{19thapril244}.
Since, $\Phi$ is a categorical equivalence,  $(~,~)^{1,0}= \xi^{1,0}$, i.e.,
\begin{equation} \label{8thmay24jb2}
(\omega, \eta)^{1,0}=((\eta^*,\omega^*)^{1,0})^*\qquad\text{for all $\omega \in \Omega^{(1,0)}$, $\eta\in \Omega^{(0,1)}$}.
\end{equation}

   In a similar way,  using Remark \ref{29thmarch24jb5} and \eqref{31stmarch241}, we deduce that  
   \begin{equation}  \label{8thmay24jb3}
     (\widetilde{\omega}, \widetilde{\eta})^{0,1}=(((\widetilde{\eta})^*, (\widetilde{\omega})^*)^{0,1})^*
     \qquad\text{for all $\widetilde{\omega} \in \Omega^{(0,1)}$, $\widetilde{\eta}\in \Omega^{(1,0)}$}.
\end{equation}
  Then,   by combining \eqref{8thmay24jb2} and \eqref{8thmay24jb3}  with \eqref{18thapril24jb1},
we get 
\[
  (\omega,\eta)_{\Omega^1}=(\eta^*,\omega^*)_{\Omega^1}^*\qquad\text{for all $\omega, \eta \in \Omega^1$}.
\]
Therefore,  Corollary \ref{9thmay24jb2} proves that the metric is real.
  \end{proof}

We end this section with a partial converse to Theorem \ref{2ndapril242} in the case when $\Omega^{(1,0)}$ and $\Omega^{(0,1)}$ are non-isomorphic simple objects in $\modz{A}{B}$. 

\begin{prop} \label{27thjuly241}
Let $(\Omega^{(\bullet, \bullet)}, \wedge, \partial, \overline{\partial})$ be an $A$-covariant complex structure on a quantum homogeneous space $B = A^{\co(H)}$ where $H$ is a Hopf $\ast$-algebra, $\Omega^1$ is an object of $\modz{A}{B}$ such that $\Omega^{(1,0)}$ and $\Omega^{(0,1)}$ are non-isomorphic simple objects in $\modz{A}{B}$. 

Then any $A$-covariant metric $(g, (~,  ~))$ on $\Omega^1$ satisfies \eqref{18thapril24jb1} and $g \in \Omega^{(1,0)} \otimes_B \Omega^{(0,1)} \oplus \Omega^{(0,1)} \otimes_B \Omega^{(1,0)}. $
\end{prop}
\begin{proof}
The map $(~ , ~): \Omega^{(1,0)} \otimes_B \Omega^{(1,0)} \rightarrow B $ is a morphism of the category $\modz{A}{B}$. Indeed, it belongs to
\begin{align*}
\Hom (\Omega^{(1,0)} \otimes_B \Omega^{(1,0)}, B)
&\cong  \Hom (\Omega^{(1,0)},  \prescript{\ast}{}{(\Omega^{(1,0)})})\quad\text{(by~\cite[(2.52)]{EtingofTensorCat})}\\
&\cong \Hom (\Omega^{(1,0)}, \Omega^{(0,1)}) \quad\text{(by Corollary~\ref{29thmarch24jb4})}\\
&= 0
\end{align*}
by Schur's lemma. Similarly, the restriction of $(~ , ~)$ to $\Omega^{(0,1)} \otimes_B \Omega^{(0,1)}$ is also zero. Thus, \eqref{18thapril24jb1} is satisfied. 

Next, since $g = \coev (1) $ for a morphism $\coev: B \rightarrow \Omega^1 \otimes_B \Omega^1 $ in $\modz{A}{B}$ (see Remark \ref{24thjuly241}), similar applications of Schur's lemma proves the assertion for $g$.
\end{proof}

\section{Hermitian metrics and their construction} \label{10thaugust243}

If~$B$ is an~$A$-comodule $\ast$-algebra, then from Example~\ref{2nddec231}, we know that $\qMod{}{B}{}{B}$ and $\qMod{A}{B}{}{B}$ are  bar-categories. Moreover, if~$M$ is an~object of $\qMod{A}{B}{}{B}$, then $\prescript{}{B}{\Hom(M,B)}$ is also an~object of $\qMod{A}{B}{}{B}$ by Proposition~\ref{12thdec236}. We will use these facts tacitly throughout this section.

\begin{defn} (\cite[Definition 4.1]{beggs2011compatible}) \label{4thdec231}
Let $B$ be a~$\ast$-algebra and $\mathcal{E}$ be a~$B$-bimodule. A~(nondegenerate) Hermitian metric on~$\mathcal{E}$ is an~isomorphism $\mathscr{H}:\overline{\mathcal{E}} \to \prescript{}{B}{\Hom(\mathcal{E},B)}$ in the category $\qMod{}{B}{}{B}$ such that 
\begin{align*}
  \langle f, \overline{e} \rangle^\ast = \langle e, \overline{f} \rangle \quad \text{for all $e, f \in \mathcal{E}$},
\end{align*}
where the map
\begin{equation} \label{3rddec231}
  \text{$\langle ~ , ~ \rangle: \mathcal{E} \otimes_B \overline{\mathcal{E}} \rightarrow B$
    is defined as  $\langle e , \overline{f} \rangle:= \ev (e \otimes_B \mathscr{H} (\overline{f}))$}.
\end{equation}
If $B$ is a~left $A$-comodule $\ast$-algebra and $\mathcal{E}$ is an~object of~$\qMod{A}{B}{}{B}$, then an~Hermitian metric~$\mathscr{H}$ on~$\mathcal{E}$ is called $A$-covariant if~$\mathscr{H}$ is an~$A$-covariant map from~$\overline{\mathcal{E}}$ to~$\prescript{}{B}{\Hom(\mathcal{E},B)}$.
\end{defn}

In the above definition, a $B$-bilinear map $\langle \cdot, ~ \cdot \rangle: \mathcal{E} \otimes_B \overline{\mathcal{E}} \rightarrow B $ was constructed from an Hermitian metric $\mathscr{H}: \overline{\mathcal{E}} \rightarrow \prescript{}{B}{\Hom(\mathcal{E},B)}$. Conversely, we have the following:

\begin{prop} \label{15thdec231}
Let $B$ be a $\ast$-algebra and let $\mathcal{E}$ be a $B$-bimodule. Suppose, we have a $B$-bilinear map
$ \langle \cdot, ~ \cdot \rangle: \mathcal{E} \otimes_B \overline{\mathcal{E}} \rightarrow B $
such that $ \langle e, \overline{f} \rangle = (\langle f, \overline{e} \rangle)^* $ and moreover, the map
$\mathscr{H}: \overline{\mathcal{E}} \rightarrow \prescript{}{B}{\Hom(\mathcal{E},B)}$ defined  by $ \mathscr{H} (\overline{f}) (e) = \langle e, \overline{f} \rangle $ is invertible, then $\mathscr{H}$ is an Hermitian metric on $\mathcal{E}$. 

Moreover, if $B$ is an $A$-comodule $\ast$-algebra and $\mathcal{E}$ an object of $\qMod{A}{B}{}{B}$ and $\langle \cdot, ~ \cdot \rangle $ is $A$-covariant, then $\mathscr{H}$ is $A$-covariant.
\end{prop}
\begin{proof} Using the $B$-bilinearity of $\langle \cdot, ~ \cdot \rangle$ and the $B$-bimodule structures of $\prescript{}{B}{\Hom(\mathcal{E},B)}$  and $\overline{\mathcal{E}}$  (see \eqref{19thdec234}  and  \eqref{28thnov231}), it follows that $\mathscr{H}$ is $B$-bilinear. Now, from the definition of $\mathscr{H}$, it is  clear that $\mathscr{H}$ is an Hermitian metric. 

For the second assertion, we assume that $\langle \cdot, ~ \cdot \rangle$ is $A$-covariant. Then for $e, f \in \mathcal{E}$, we have 
\begin{eqnarray*}
\prescript{B}{}{\delta} (\mathscr{H} (\overline{f}) (e)) &=&  \prescript{B}{}{\delta} (\langle e, \overline{f} \rangle)\\
&=& (\id \otimes \langle \cdot, ~ \cdot \rangle) \prescript{\mathcal{E} \otimes_B \overline{\mathcal{E}}}{}{\delta} (e \otimes_B \overline{f})\\
&=& (\id \otimes \langle \cdot, ~ \cdot \rangle) (e_{(- 1)} f^*_{(- 1)} \otimes e_{(0)} \otimes_B \overline{f_{(0)}})
    \quad\text{(by~\eqref{25thnov232})}\\
&=& e_{(- 1)} f^*_{(- 1)} \otimes \langle e_{(0)}, \overline{f_{(0)}} \rangle\\
&=& e_{(- 1)} f^*_{(- 1)} \otimes \mathscr{H} (\overline{f_{(0)}}) (e_{(0)}).
\end{eqnarray*}
Using this, for any $e$ in $\mathcal{E}$, we obtain
\begin{eqnarray*}
\prescript{\prescript{}{B}{\Hom(\mathcal{E},B)}}{}{\delta} (\mathscr{H} (\overline{f})) (e) &=& S (e_{(- 1)}) [ \mathscr{H} (\overline{f}) (e_{(0)}) ]_{(- 1)} \otimes [ \mathscr{H} (\overline{f}) (e_{(0)}) ]_{(0)}\quad\text{(by~\eqref{27thnov233})}\\
&=& S (e_{(- 2)}) e_{(- 1)} f^*_{(- 1)} \otimes \mathscr{H} (\overline{f_{(0)}}) (e_{(0)})\\
&=& f^*_{(- 1)} \otimes \mathscr{H} (\overline{f_{(0)}}) (e)\\
&=& (\id \otimes \mathscr{H}) \prescript{\overline{\mathcal{E}}}{}{\delta} (\overline{f}) (e).
\end{eqnarray*}
This completes the proof of the proposition.
\end{proof}

\subsection{Construction of Hermitian metrics on quantum homogeneous spaces}

Proposition 8.30 of \cite{BeggsMajid:Leabh} gives a prescription to construct an Hermitian metric from a metric in the sense of Definition \ref{4thmay242}. In fact, we have the following result.

\begin{prop} \label{19thapril241}
        Let $(\Omega^{\bullet}, \wedge, d)$ be an $A$-covariant $\ast$-differential calculus on a left $A$-comodule $\ast$-algebra $B$ such that $\Omega^1  $ is finitely generated and projective as a left $B$-module. Then there is a one to one correspondence between real $A$-covariant metrics (in the sense of Definition \ref{3rdjuly244})  and $A$-covariant Hermitian metrics on $\Omega^1$. 
    \end{prop}
    \begin{proof}
    For the proof, we will use the fact that the map $\star_{\Omega^1}: \Omega^1 \rightarrow \overline{\Omega^1} $ of \eqref{8thjuly241} is a left $A$-covariant $B$-bilinear map.
    
Let $(g, (~,~))$ be a real $A$-covariant metric.  We define 
$$\langle~,~\rangle:  \Omega^1  \otimes_B\overline{\Omega^1} \to B ~ \text{by} ~ \langle\omega, \overline{\eta}\rangle=(\omega, \eta^*).$$
Since  $\langle~,~\rangle = (~ , ~) \circ (\id \otimes_B  \star_{\Omega^1}^{-1}), $ the map  $\langle~,~\rangle $ is $A$-covariant. Since the metric is real, we have $(\omega, \eta)^*=(\eta^*, \omega^*)$ for all $\omega, \eta \in \Omega^1$ by Corollary \ref{9thmay24jb2}. Therefore,  
$$ \langle \omega, \overline{\eta} \rangle = (\langle \eta, \overline{\omega} \rangle)^* $$
for all $\omega, \eta \in \Omega^1$.
Next, we define a $B$-bilinear map
$$\mathscr{H}:\overline{\Omega^1} \to \prescript{  }{B}{\Hom(\Omega^1, B)} ~ \text{as} ~ \mathscr{H}(\overline{e})(f)=\langle f, \overline{e}\rangle.$$ 
The invertibility of the map $\mathscr{H}$ easily follows from the fact that $(g, (~,~))$ is a metric on $\Omega^1$.

Thus, Proposition \ref{15thdec231} implies that $\mathscr{H}$ is an $A$-covariant Hermitian metric.

    Conversely, given an $A$-covariant Hermitian metric $\mathscr{H}$ on $\Omega^1$, we define $g$ and $(~,~)$ by $g:= \sum^n_{i = 1} \star_{\Omega^1}^{-1} \mathscr{H}^{-1} (e_i) \otimes_B e^i$ where $ \{ e^i, e_i: i = 1, \cdots n \}$ is a dual basis of $\Omega^1$ and $ (\omega, \eta) = \langle \omega, \overline{\eta^*}\rangle$. 
    
    Since the element $ \sum^n_{i = 1} e_i \otimes_B e^i $ is the inverse image of $\id_{\Omega^1}$ under the canonical isomorphism $ \Omega^1 \otimes_B \prescript{}{B}{\Hom (\Omega^1, B) } \rightarrow \prescript{}{B}{\Hom (\Omega^1, \Omega^1)}, $ the definition of $g$ is independent of the choice of the dual basis.  It can be easily checked that
    $ (g, (~, ~)) $ is a covariant metric on $\Omega^1. $
    \end{proof}
    
Now we prove that under some conditions on the calculus, the covariant Hermitian metric on $\Omega^1$ constructed from a real covariant metric has a special form.

 \begin{thm} \label{9thmay24jb21}
    Suppose $(\Omega^{(\bullet,\bullet)}, \wedge, \partial, \overline{\partial})$ is an $A$-covariant complex structure on a quantum homogeneous space $B=A^{\co(H)}$, where $A$ is a Hopf $\ast$-algebra and $\Omega^1$ is an object of $\modz{A}{B}$.  Moreover, we assume that $\Omega^{(1,0)}$ and $\Omega^{(0,1)}$ are non isomorphic simple objects in $\modz{A}{B}$.

    If $(g, (~, ~))$ is an $A$-covariant real metric on $\Omega^1$, then by the prescription of Proposition \ref{19thapril241}, we have an $A$-covariant Hermitian metric $  \mathscr{H} $ on $\Omega^1$. Moreover, if we define
    \begin{equation*}
            \mathscr{H}_1:\overline{\Omega^{(1,0)}} \to \prescript{}{B}{ \Hom (\Omega^{(1,0)}, B)} \quad \text{by} \quad \mathscr{H}_1(\overline{\omega})(\eta)= (\eta, \omega^*) ~ \text{and}
        \end{equation*}
        
        \begin{equation*}
           \mathscr{H}_2:\overline{\Omega^{(0,1)}} \to \prescript{}{B}{ \Hom (\Omega^{(0,1)}, B)} \quad \text{by} \quad \mathscr{H}_2(\overline{\omega})(\eta)=  (\eta, \omega^*),
        \end{equation*}
        then $\mathscr{H}_1$ and   $\mathscr{H}_2$ are $A$-covariant Hermitian metrics  on $ \Omega^{(1,0)} $ and $\Omega^{(0,1)}$ respectively and we have
				$\mathscr{H} = \mathscr{H}_1 \oplus \mathscr{H}_2$.
    \end{thm}
\begin{proof} 
We begin by observing that by virtue of  Remark \ref{27thsep241}, $\Omega^1$ is projective as a left $B$-module.
The fact that $ \mathscr{H} $ is indeed an $A$-covariant Hermitian metric follows directly from Proposition \ref{19thapril241}. 

For the assertions regarding $\mathscr{H}_1$ and $\mathscr{H}_2$, we start by making some observations using Proposition \ref{27thjuly241}. Indeed, by that result, 
 $g \in \Omega^{(1,0)} \otimes_B \Omega^{(0,1)} \oplus \Omega^{(0,1)} \otimes_B \Omega^{(1,0)}$.
Let us use Sweedler like notation to write 
$$ g = X_{(1)} \otimes_B X_{(2)} + Y_{(1)} \otimes_B Y_{(2)}, $$
where 
$ X_{(1)} \otimes_B X_{(2)} \in \Omega^{(0, 1)} \otimes_B \Omega^{(1,0)} $ and  $ Y_{(1)} \otimes_B Y_{(2)} \in \Omega^{(1, 0)} \otimes_B \Omega^{(0,1)}. $
Then if $\omega \in \Omega^{(0, 1)}$, we have
\begin{equation} \omega = g_{(1)} (g_{(2)}, \omega) = X_{(1)} (X_{(2)}, \omega)  +  Y_{(1)} (Y_{(2)}, \omega)  = X_{(1)} (X_{(2)}, \omega)  \label{19thapril242} 
\end{equation}
as the equation \eqref{18thapril24jb1} is satisfied by Proposition \ref{27thjuly241}.  Similarly, if $\omega \in \Omega^{(1,0)}$, then
\begin{equation} \label{9thjuly241}
(\omega, X_{(1)}) X_{(2)} = \omega.
\end{equation}

For the second statement, it is clear that it is enough to prove that $\mathscr{H}_1$ and $\mathscr{H}_2$ are  isomorphisms.
If $\mathscr{H}_1(\overline{x})=0$ for some $x \in \Omega^{(1,0)}$ so that  $ (y, x^*) =0$ for all $y\in \Omega^{(1,0)}$,  then
 $$0  = X_{(1)}(X_{(2)}, x^*) = x^* $$
by \eqref{19thapril242}, as $x^* \in \Omega^{(0, 1)}. $ Thus, $x = 0$ proving that $\mathscr{H}_1$ is one-one.

        Now suppose $f\in \prescript{}{B}{ \Hom (\Omega^{(1,0)}, B)}$. Let $x= f(X_{(2)})^* X_{(1)}^* \in \Omega^{(1, 0)}. $ Then for all $y\in \Omega^{(1,0)}$, we have
$$ \mathscr{H}_1(\overline{x})(y)=  (y, x^*)  = (y, X_{(1)})f(X_{(2)}) = f((y, X_{(1)}) X_{(2)})  = f(y)$$
        by \eqref{9thjuly241} and the left $B$-linearity of $f$. This proves that $\mathscr{H}_1$ is onto.

Similarly,  $\mathscr{H}_2$ is also an $A$-covariant Hermitian metric on $ \Omega^{(0, 1)}$. This completes the proof. 
\end{proof}

\section{The Chern connection as a bimodule connection} \label{4thaugust243}

Let $(\mathcal{E}, \overline{\partial}_{\mathcal{E}}) $ be a holomorphic $B$-bimodule (in the sense of Definition \ref{11thdec231}) on a $\ast$-algebra $B$ equipped with  an Hermitian metric $\mathscr{H}$. Then there exists a unique connection $\nabla$ on $\mathcal{E}$ which is compatible with $\mathscr{H}$ (see Definition \ref{19thoct231}) and whose $(0, 1)$-part coincides with $\overline{\partial}_{\mathcal{E}}$. This was proved by Beggs and Majid in \cite{BeggsMajidChern}. The connection $\nabla$ is called the Chern connection for the pair $(\mathcal{E}, \mathscr{H})$.

In this section, we start by observing  that if a complex structure $(\Omega^{(\bullet, \bullet)}, \wedge, \partial, \overline{\partial})$ on $B$ is factorizable, then the $\overline{\partial}$-connection $\overline{\partial}_{\Omega^{(1, 0)}}$ for the holomorphic $B$-bimodule $(\Omega^{(1, 0)}, \overline{\partial}_{\Omega^{(1, 0)}}) $ of Proposition \ref{25thsep232} is in fact a bimodule $\overline{\partial}$-connection. The second assertion of Theorem \ref{chern} then implies that the Chern connection $\nabla_{\Ch}$ for any Hermitian metric $\mathscr{H}_1$ on $\Omega^{(1, 0)}$ is a bimodule connection. It follows that under the factorizability assumption, this result continues to hold for a Chern connection $\nabla_{\Ch, ~ \op}$ on $\Omega^{(0, 1)}$ for the opposite complex structure and an Hermitian metric $\mathscr{H}_2$ on $\Omega^{(0, 1)}$. As a result, if $\mathscr{H}_1$ and $\mathscr{H}_2$ are Hermitian metrics on $\Omega^{(1, 0)}$ and $\Omega^{(0, 1)}$ respectively, then the  Chern connection $\nabla_{\Ch}$ for the pair $(\Omega^{(1, 0)}, \mathscr{H}_1)$ extends to a bimodule connection $\nabla_{\Ch} + \nabla_{\Ch, \op}$ on $\Omega^1$. 

\subsection{Definition and existence of the Chern connection} \label{4thaugust242}

We start by recalling the definition of compatibility of a connection on a bimodule $\mathcal{E}$ with an Hermitian metric.

\begin{defn} (\cite[Definition 8.33]{BeggsMajid:Leabh}) \label{19thoct231}
 Let $\nabla$ be a left connection on a $B$-bimodule $\mathcal{E}$, where $B$ is a $\ast$-algebra equipped with a $\ast$-differential calculus. Then $\nabla$ is said to be compatible with an Hermitian metric $\mathscr{H}$ on $\mathcal{E}$   if
\begin{align*}
    d\langle ~ , ~ \rangle = (\id \otimes_B \langle ~ , ~\rangle)(\nabla \otimes_B \id) + (\langle ~ , ~ \rangle \otimes_B \id) (\id\otimes_B \widetilde{\nabla})
\end{align*}
as maps from $ \mathcal{E} \otimes_B \overline{\mathcal{E}} $ to $ \Omega^1$.

 Here, $\langle ~ , ~ \rangle : \mathcal{E} \otimes_B \overline{\mathcal{E}} \to B $ is the map introduced in \eqref{3rddec231} while $\widetilde{\nabla}: \overline{\mathcal{E}}\to \overline{\mathcal{E}}\otimes_B \Omega^1$ is the right connection constructed from $\nabla$ by defining 
 $$\widetilde{\nabla} (\overline{e})= \sum_i \overline{x_i}\otimes_B \omega_i^{\ast},$$
 where $\nabla (e)= \sum_i \omega_i \otimes_B x_i $.
\end{defn}

If $\nabla$ is a left connection on a $B$-bimodule $\mathcal{E}$ and $(\Omega^{(\bullet, \bullet)}, \wedge, \partial, \overline{\partial})$ is a complex structure on $B$, then the $(1, 0)$-part of the connection $\nabla$ is defined as:
$$ \nabla^{(1, 0)}: = (\pi^{(1, 0)} \otimes_B \id) \nabla: \mathcal{E} \rightarrow \Omega^{(1, 0)} \otimes_B \mathcal{E}, $$
while the $(0, 1)$-part is defined as
$$ \nabla^{(0, 1)}: = (\pi^{(0, 1)} \otimes_B \id) \nabla: \mathcal{E} \rightarrow \Omega^{(0, 1)} \otimes_B \mathcal{E}.$$

The following theorem of Beggs and Majid generalizes the existence of Chern connection on holomorphic Hermitian vector bundles on complex manifolds.

\begin{thm} \label{chern}  \cite[Proposition 4.2 and Proposition 4.3]{BeggsMajidChern} 
 Suppose $ (\Omega^{(\bullet, \bullet)}, \wedge, \partial, \overline{\partial}) $ is a complex structure on a $\ast$-algebra $B$ and $\mathscr{H}$ an Hermitian metric on a holomorphic $B$-bimodule $ (\mathcal{E}, \overline{\partial}_{\mathcal{E}})$, where $\mathcal{E}$ is finitely generated and projective as a left $B$-module. Then the following statements hold:
\begin{enumerate}
\item There exists a unique left connection $\nabla_{\bf\Ch}$ on $\mathcal{E}$ called the Chern connection, which is compatible with $\mathscr{H}$ (in the sense of Definition \ref{19thoct231}) and whose $(0, 1)$-part coincides with $\overline{\partial}_{\mathcal{E}}$. 

Hence, we have
\begin{equation} \label{19thoct233}
\nabla_{\Ch} = (\pi^{(1, 0)} \otimes_B \id) \nabla_{\Ch} + \overline{\partial}_{\mathcal{E}}.
\end{equation}

\item If $\sigma_{\mathcal{E}, (0, 1)}: \mathcal{E} \otimes_B \Omega^{(0, 1)} \rightarrow \Omega^{(0, 1)} \otimes_B \mathcal{E} $ is a $B$-bimodule map such that 
$(\overline{\partial}_{\mathcal{E}}, \sigma_{\mathcal{E}, (0, 1)})$ is a left bimodule $\overline{\partial}$-connection on $\mathcal{E}$, then there exists a $B$-bimodule map
  $$ \sigma_{\mathcal{E}, (1, 0)}: \mathcal{E} \otimes_{B} \Omega^{(1, 0)} \rightarrow \Omega^{(1, 0)} \otimes_{B} \mathcal{E}$$
	such that $(\nabla_{\Ch}, \sigma_{\mathcal{E}, (1, 0)} \oplus \sigma_{\mathcal{E}, (0, 1)})$ is a bimodule connection. 
\end{enumerate}
\end{thm}

For the explicit expression of $\sigma_{\mathcal{E}, (1, 0)}$ in terms of $\sigma_{\mathcal{E}, (0, 1)}$ and $\mathscr{H}$, we refer to \cite[Proposition 4.3]{BeggsMajidChern}.

The proof of the next theorem is the content of Section \ref{1staugust241}. Since the proof uses various notations and  identities which are not used in the rest of the article, we have given the proof in a separate section.  

\begin{thm} \label{23rdoct231} 
Let $ (\Omega^{(\bullet, \bullet)}, \wedge, \partial, \overline{\partial}) $ be an $A$-covariant  complex structure on a $\ast$-algebra $B$. If $ (\mathcal{E}, \overline{\partial}_{\mathcal{E}})$ is a covariant holomorphic $B$-bimodule equipped with an $A$-covariant Hermitian metric $\mathscr{H}$ with $\mathcal{E}$ finitely generated and projective as a left $B$-module, then the Chern connection
$\nabla_{\Ch}$ is a covariant connection.
\end{thm}
\begin{proof}
See Theorem \ref{30thjuly241}.
\end{proof}

\subsection{The Chern connection and factorizable complex structures}

\begin{defn} \label{13thdec231}
 A complex structure $ (\Omega^{(\bullet, \bullet)}, \wedge, \partial, \overline{\partial}) $ over a $\ast$-algebra $B$ is called factorizable if  
the restriction of the wedge maps 
$$ \wedge_{(0, q), (p,  0)}: \Omega^{(0, q)} \otimes_{B} \Omega^{(p, 0)} \rightarrow \Omega^{(p,q)} ~ \text{and} ~  \wedge_{(p, 0), (0,  q)}: \Omega^{(p, 0)} \otimes_{B} \Omega^{(0, q)} \rightarrow \Omega^{(p,q)}      $$ 
are invertible for all $ (p, q) \in \mathbb{N} \times \mathbb{N}. $
\end{defn}

In what follows, we will denote the inverses of the maps $\wedge_{(0, q), (p,  0)}$ and $ \wedge_{(p, 0), (0,  q)} $ by the symbols $\theta^{(p,q)}_l$ and $\theta^{(p, q)}_r$ respectively. Thus, if the differential calculus is factorizable, then
\begin{equation} \label{25thsep231} 
\theta^{(p, q)}_l: \Omega^{(p,q)} \rightarrow \Omega^{(0, q)} \otimes_{B} \Omega^{(p, 0)} ~ \text{and} ~  \theta^{(p, q)}_r: \Omega^{(p,q)} \rightarrow \Omega^{(p, 0)} \otimes_{B} \Omega^{(0, q)}  \end{equation}
are $B$-bimodule isomorphisms and the following equations hold:
\begin{equation} \label{19thoct236}
\wedge_{(0, q), (p, 0)} \circ \theta^{(p, q)}_l = \id_{\Omega^{(p, q)}} = \wedge_{(p, 0), (0, q)} \circ \theta^{(p, q)}_r.
\end{equation}

Factorizable complex structures endow $\Omega^{(1,0)}$ with a canonical holomorphic structure.

\begin{prop} \label{25thsep232}  \cite[Proposition 6.1]{BS}
If $ (\Omega^{(\bullet, \bullet)}, \wedge, \partial, \overline{\partial}) $ is a factorizable complex structure on a $\ast$-algebra $B$ as in Definition \ref{13thdec231} and $\theta^{(1, 1)}_l$ is the map as in \eqref{25thsep231}, then
\begin{equation} \label{19thoct232}
\overline{\partial}_{\Omega^{(1,  0)}}:= \theta^{(1, 1)}_l \circ \overline{\partial}: \Omega^{(1,  0)} \rightarrow  \Omega^{(0,  1)} \otimes_{B} \Omega^{(1,  0)}
\end{equation}
is a $\overline{\partial}$-connection with zero holomorphic curvature. 

Thus, the pair $ (\Omega^{(1, 0)}, \overline{\partial}_{\Omega^{(1,  0)}}) $ is a holomorphic $B$-bimodule if the complex structure is factorizable. 
\end{prop}

\begin{cor} \label{17thjuly241}
Let $ (\Omega^{(\bullet, \bullet)}, \wedge, \partial, \overline{\partial}) $ be an $A$-covariant factorizable complex structure on a $\ast$-algebra $B$ such that $\Omega^1  $ is finitely generated and projective as a left $B$-module and $(\Omega^{(1,0)}, \overline{\partial}_{\Omega^{(1,0)}}) $ the holomorphic structure on $\Omega^{(1,0)}$ from Proposition \ref{25thsep232}. Then the Chern connection on $\Omega^{(1, 0)}$ for any $A$-covariant Hermitian metric on $\Omega^{(1,0)}$ is $A$-covariant.
\end{cor}
\begin{proof} As the differential calculus is $A$-covariant,  the map  $\wedge_{(0,1), (1, 0)}$ is $A$-covariant so that $\theta^{(1, 1)}_l$ is also $A$-covariant. Moreover, since the complex structure is $A$-covariant,  $\overline{\partial}$ is $A$-covariant and hence, $\overline{\partial}_{\Omega^{(1,  0)}} = \theta^{(1, 1)}_l \circ \overline{\partial} $ is $A$-covariant.

Therefore, $(\Omega^{(1,0)}, \overline{\partial}_{\Omega^{(1,0)}}) $ is a covariant holomorphic $B$-bimodule and hence, the corollary follows from Theorem \ref{23rdoct231}.
\end{proof}

Next we establish  the following result, whose proof is straightforward. We will use the symbol $\theta^{(1, 1)}_l$ defined in \eqref{25thsep231}.

\begin{prop} \label{20thoct232}
Let $ (\Omega^{(\bullet, \bullet)}, \wedge, \partial, \overline{\partial}) $ be a factorizable complex structure over a $\ast$-algebra $B$  with $\Omega^1 $  finitely generated and projective as a left $B$-module  and $\overline{\partial}_{\Omega^{(1, 0)}}$ be the  $\overline{\partial}$-connection  of Proposition \ref{25thsep232}. 

Consider the $B$-bimodule map
$$ \sigma_{(1,0), (0, 1)}:= - \theta^{(1, 1)}_l \circ \wedge_{(1,0), (0, 1)}: \Omega^{(1, 0)} \otimes_B \Omega^{(0, 1)} \rightarrow \Omega^{(0, 1)} \otimes_B \Omega^{(1, 0)}.  $$
Then the following statements hold:
\begin{enumerate}
\item The pair $ (\overline{\partial}_{\Omega^{(1, 0)}},  \sigma_{(1,0), (0, 1)}) $ is a left bimodule $\overline{\partial}$-connection.

\item Suppose $\mathscr{H}_1$ is an Hermitian metric on $\Omega^{(1,0)}$. Then there exists a $B$-bimodule map
$$ \sigma_{(1,0), (1, 0)}: \Omega^{(1, 0)} \otimes_B \Omega^{(1,0)} \rightarrow \Omega^{(1, 0)} \otimes_B \Omega^{(1,0)}  $$
such that the Chern connection $\nabla_{\Ch}$ on $ (\Omega^{(1, 0)}, H_1)$ is a $\sigma_{ (1, 0), (1, 0)} \oplus \sigma_{(1, 0), (0, 1)}$-bimodule connection. 
\end{enumerate}
\end{prop}
\begin{proof}
    The first assertion was already observed in  \cite[Proposition 7.33]{BeggsMajid:Leabh}.
    The second statement now follows from the second assertion of Theorem \ref{chern}.
\end{proof}

\subsection{The opposite complex structure on the calculus} \label{17thjuly242}

If $ (\Omega^{(\bullet, \bullet)} \wedge, \partial, \overline{\partial}) $ is a complex structure on $B$, then the opposite complex structure on the calculus is given by the $\bN^2_0$-graded algebra
 $\bigoplus_{(a,b)\in \bN^2_0} \Om^{(a,b), \op}$, where
$$ \Om^{(a,b), \op}:= \Omega^{(b, a)}.  $$
Consequently, the $\partial$ and $\overline{\partial}$-operators of the opposite complex structure are given by:
$$ \partial_{\op}:= \overline{\partial}: \Omega^{(a, b), \op} \rightarrow \Omega^{(a + 1, b), \op}, ~  \overline{\partial}_{\op}:= \partial: \Omega^{(a, b), \op} \rightarrow \Omega^{(a, b + 1), \op}. $$
In the sequel, we will denote this opposite complex structure by the quadruple $(\Omega^{(\bullet, \bullet) \op}, \wedge, \partial_{\op}, \overline{\partial}_{\op}). $

Let $ (\Omega^{(\bullet, \bullet)}, \wedge, \partial, \overline{\partial}) $ be a factorizable complex structure over a $\ast$-algebra $B$ with $\Omega^1 $ finitely generated and projective as a left $B$-module. Then by a combination of Proposition \ref{25thsep232} and Proposition \ref{20thoct232}, we have a similar result for the opposite complex structure. Indeed, let us define a map
$$ \overline{\partial}_{\Omega^{(1,0)}, \op}: \Omega^{(0,1)} \rightarrow \Omega^{(1,0)} \otimes_B \Omega^{(0, 1)} ~ \text{by} ~ \overline{\partial}_{\Omega^{(1,0)}, \op} = \theta^{(1, 1)}_r \circ \overline{\partial_\op}, $$
where $\theta^{(1, 1)}_r$ is defined in \eqref{25thsep231}.

 Then the pair $ (\Omega^{(0, 1)}, \overline{\partial}_{\Omega^{(1,0)}, ~ \op})$ is a holomorphic $B$-module with respect to the opposite complex structure $ (\Omega^{(\bullet, \bullet) \op}, \wedge, \partial_{\op}, \overline{\partial}_{\op}). $ Moreover, if $\sigma_{(0, 1), (1, 0)}: \Omega^{(0, 1)} \otimes_B \Omega^{(1, 0)} \rightarrow \Omega^{(1, 0)} \otimes_B \Omega^{(0, 1)} $ is defined as 
$$ \sigma_{(0, 1), (1, 0)}:= - \theta^{(1, 1)}_r \circ \wedge_{(0, 1), (1, 0)},$$
  then $ (\overline{\partial}_{\Omega^{(1, 0)}, \op},  \sigma_{(0,1), (1, 0)}) $ is a left bimodule $\overline{\partial}^{\op}$-connection.

 If $\mathscr{H}_2$ is an Hermitian metric on $\Omega^{(1,0), \op}$, then  the Chern connection  for the pair  $ (\Omega^{(1, 0), \op}, \mathscr{H}_2)$ exists and is a bimodule connection, to be denoted by $ \nabla_{\Ch, \op}. $

Summarizing, we have the following result, where we use the fact that the sum of two bimodule connections is again a bimodule connection.

\begin{thm} \label{23rdoct236}
Let $ (\Omega^{(\bullet, \bullet)}, \wedge, \partial, \overline{\partial}) $ be a factorizable complex structure over a $\ast$-algebra $B$ as above such that $\Omega^1$ is finitely generated and projective as a left $B$-module. Let $\mathscr{H}_1$ and $\mathscr{H}_2$ be Hermitian metrics on $\Omega^{(1,0)}$ and $\Omega^{(0, 1)} = \Omega^{(1, 0), \op} $ respectively. If $ \nabla_{\Ch} $ and $\nabla_{\Ch,  \op}$ denote the Chern connections for the pairs $ (\Omega^{(1, 0)}, \mathscr{H}_1) $  and $(\Omega^{(1, 0), \op}, \mathscr{H}_2) $ respectively, then the connection
$$\nabla: = \nabla_{\Ch} + \nabla_{\Ch,  \op}: \Omega^1 \rightarrow \Omega^1 \otimes_B \Omega^1 $$
is a bimodule connection on $\Omega^1$. 
\end{thm}

\vspace{4mm}

We end this section by observing that the conclusion of Theorem \ref{23rdoct236} is applicable to the set up of Theorem \ref{9thmay24jb21}. 

\begin{cor} \label{27thjuly243}
 Suppose $(\Omega^{(\bullet,\bullet)}, \wedge, \partial, \overline{\partial})$ is a factorizable $A$-covariant complex structure on a quantum homogeneous space $B=A^{\co(H)}$ satisfying the hypotheses of Theorem  \ref{9thmay24jb21}. Let $(g, ( ~, ~ ) )$ be a real $A$-covariant metric on $\Omega^1$ and $ \mathscr{H}_1, \mathscr{H}_2$ be the $A$-covariant Hermitian metrics on $\Omega^{(1,0)}$ and $\Omega^{(0,1)}$ constructed in Theorem  \ref{9thmay24jb21}. Then the connection $\nabla = \nabla_{\Ch} + \nabla_{\Ch,  \op}$ of Theorem \ref{23rdoct236} is an $A$-covariant bimodule connection on $\Omega^1$.
 \end{cor}
\begin{proof}
We will use the fact that $\Omega^1$ is projective as a left $B$-module which follows from Remark \ref{27thsep241}. 
By Corollary \ref{17thjuly241}, we know that both $ \nabla_{\Ch} $ and $  \nabla_{\Ch,  \op} $ are $A$-covariant. Moreover,  Theorem \ref{23rdoct236} implies that $\nabla$ is a bimodule connection. This completes the proof.
\end{proof}

\section{Levi-Civita connections for the Heckenberger--Kolb calculi} \label{4thaugust244}

This section is about the Heckenberger--Kolb calculi for the quantized irreducible flag manifolds. We apply the results obtained in the previous sections (in particular Theorem \ref{8thmay24jb5}) to classify covariant metrics on these calculi. In particular, we show that there exists a unique (up to scalar) real covariant quantum symmetric metric on the bimodule of one-forms of the Heckenberger--Kolb calculi. Thus, this particular quantum symmetric metric agrees with the generalized Fubini-Study metric on quantum projective spaces constructed by Matassa in  \cite{matassalevicivita}. Finally, by using the results of Section \ref{4thaugust243}, we prove the existence and uniqueness of Levi-Civita connection for any covariant real metric on the Heckenberger--Kolb calculi. This generalizes the existence and uniqueness theorem proved by Matassa for quantum projective spaces.

We begin by recalling the relevant definitions of the quantized universal enveloping algebra, the corresponding quantum co-ordinate algebras and in particular, the quantized irreducible flag manifolds. 

\subsection{Drinfeld--\/Jimbo quantum groups}

Let $\mathfrak{g}$ be a finite-dimensional complex simple Lie algebra of rank $r$. We fix a Cartan subalgebra $\mathfrak{h}$ with corresponding root system $\Delta \subseteq \mathfrak{h}^{\ast}$, where $\mathfrak{h}^{\ast}$ denotes the
linear dual of $\mathfrak{h}$. With respect to a choice of simple roots $\Pi= \{ \alpha_1, \alpha_2, \cdots \alpha_r\}$, denote by $(\cdot, \cdot)$ the symmetric bilinear form induced on $\mathfrak{h}^{\ast}$ by the Killing form of $\mathfrak{g}$, normalized
so that any shortest simple root $\alpha_i$ satisfies $(\alpha_i, \alpha_i)=2$. The coroot $\alpha_i^\vee$ of a simple
root $\alpha_i$ is defined by
\begin{align*}
    \alpha^\vee_i:=\frac{ \alpha_i}{d_i}= 2\frac{\alpha_i}{(\alpha_i, \alpha_i)}, \quad\text{where }  d_i=\frac{(\alpha_i, \alpha_i)}{2}.
\end{align*}
The Cartan matrix $\mathcal{A} = (a_{i j})_{i j}$ of $\mathfrak{g}$ is the $(r \times r)$-matrix defined by $a_{ij}:= (\alpha^\vee_i, \alpha_j)$. Let $\{\varpi_1,\cdots \varpi_r\}$ denote the corresponding set of fundamental weights of $\mathfrak{g}$, which is to say, the dual basis of the coroots.

Let $q \in  \mathbb{R}$ such that $q \notin \{-1, 0, 1\}$, and denote $q_i:=q^{d_i}$. The quantized universal enveloping algebra $U_q(\mathfrak{g})$ is the noncommutative associative algebra generated by the
elements $E_i, F_i, K_i$, and $K_i^{-1}$, for $i = 1,\cdots,r$, 
satisfying the relations (12) - (16) of  \cite[Chapter 6]{KSLeabh}.

A Hopf algebra structure is defined on $U_q (\mathfrak{g})$ by
\begin{gather*}
  \Delta E_i = E_i\otimes K_i + 1\otimes E_i,\quad
  \Delta F_i = F_i\otimes 1 + K^{-1}_i \otimes F_i,\quad
  \Delta K^{\pm}_i = K^{\pm}_i \otimes K^{\pm}_i,\\
  S(E_i) = - E_iK^{-1}_i,\quad
  S(F_i) = - K_iF_i,\quad
  S(K^{\pm}_i) = K^{\mp}_i,\\
  \epsilon (K_i) = 1,\quad \epsilon (E_i)=\epsilon(F_i)=0.
\end{gather*}

A Hopf $\ast$-algebra structure, called the compact real form of $U_q (\mathfrak{g})$, is defined by
\begin{align*}
    K^{\ast}_i:= K_i, \quad E^{\ast}_i:= K_iF_i, \quad F^{\ast}_i= E_i K^{-1}_i.
\end{align*}
Let $\mathcal{P}\subset\fh^\ast$ be the weight lattice of $\mathfrak{g}$, and $\mathcal{P}_+\subset\mathcal{P}$ its set of dominant integral weights.
For each $\mu \in \mathcal{P}_+$ there exists an irreducible finite-dimensional $U_q (\mathfrak{g})$-module~$V_{\mu}$,
uniquely defined by the existence of a~vector $v_{\mu}\in  V_{\mu}$, which we call a highest weight vector, satisfying
\begin{align*}
    E_i \triangleright v_{\mu}=0, \quad K_i\triangleright v_{\mu}= q^{(\mu,\alpha^\vee_i)}v_{\mu}, \quad \text{for all $i=1, \cdots ,r$}.
\end{align*}
Moreover, $v_{\mu}$ is the unique such element up to scalar multiple. We call any finite direct
sum of such $U_q (\mathfrak{g})$-representations a finite-dimensional type-$1$ representation. 
For further details on Drinfeld--Jimbo quantized enveloping algebras, we refer the reader to the standard
texts~\cite{ChariPressley, KSLeabh}, or to the seminal papers~\cite{DrinfeldICM, Jimbo1986}.

\subsection{Quantum Coordinate Algebras and irreducible flag manifolds} 
In this subsection we recall some necessary material about quantized coordinate algebras.
Let $V$ be a finite-dimensional left $U_q(\mathfrak{g})$-module, $v \in V$, and $f \in V^*$, the $\mathbb{C}$-linear dual of $V$, endowed with its  right \mbox{$U_q(\mathfrak{g})$-module} structure.

Consider the function  $c^{\tiny{V}}_{f,v}:U_q(\mathfrak{g}) \to \mathbb{C}$ defined by $c^{\tiny{V}}_{f,v}(X) := f\big(X \triangleright v\big)$. The \emph{space of matrix coefficients} of $V$ is the subspace
\begin{align*}
C(V) := \Span_{\mathbb{C}}\!\left\{ c^{\tiny{V}}_{f,v} \,| \, v \in V,  f \in V^*\right\} \subseteq U_q(\mathfrak{g})^*.
\end{align*}

Let $U_q(\mathfrak{g})^\circ$ denote the Hopf dual of $U_q(\mathfrak{g})$. It is easily checked that a Hopf  subalgebra of $U_q(\mathfrak{g})^{\circ}$ is given by
\begin{equation}\label{eq:PeterWeyl}
\mathcal{O}_q(G) := \bigoplus_{\mu \in \mathcal{P}^+} C(V_{\mu}).
\end{equation}
We call $\mathcal{O}_q(G)$ the {\em quantum coordinate algebra of~$G$}, where~$G$ is the compact, connected, simply-connected, simple Lie group  having~$\mathfrak{g}$ as its complexified Lie algebra.
We note that $\mathcal{O}_q(G)$ is a cosemisimple Hopf algebra by construction.

Now we recall the definition of quantized flag manifolds.

For $\{\alpha_i\}_{i\in S} \subseteq \Pi$ a subset of simple roots,  consider the Hopf $*$-subalgebra
\begin{align*}
U_q(\mathfrak{l}_S) := \big< K_i, E_j, F_j \,|\, i = 1, \ldots, r; j \in S \big>.
\end{align*} 
The category of $U_q(\mathfrak{l}_S)$-modules is known to be  semisimple. The Hopf $\ast$-algebra embedding $\iota_S:U_q(\mathfrak{l}_S) \hookrightarrow U_q(\mathfrak{g})$ induces the dual Hopf \mbox{$*$-algebra} map $\iota_S^{\circ}: U_q(\mathfrak{g})^{\circ} \to U_q(\mathfrak{l}_S)^{\circ}$. By construction $\mathcal{O}_q(G) \subseteq U_q(\mathfrak{g})^{\circ}$, so we can consider the restriction map
\begin{align*}
\pi_S:= \iota_S^{\circ}|_{\mathcal{O}_q(G)}: \mathcal{O}_q(G) \to U_q(\mathfrak{l}_S)^{\circ},
\end{align*}
and the Hopf $*$-subalgebra 
$
\mathcal{O}_q(L_S) := \pi_S\big(\mathcal{O}_q(G)\big) \subseteq U_q(\mathfrak{l}_S)^\circ.
$
The {\em quantum flag manifold associated} to $S$ is the quantum homogeneous space associated to the surjective  Hopf $*$-algebra map  $\pi_S:\mathcal{O}_q(G) \to \mathcal{O}_q(L_S)$. We denote it by
\begin{align*}
\mathcal{O}_q\big(G/L_S\big) := \mathcal{O}_q \big(G\big)^{\co\left(\mathcal{O}_q(L_S)\right)}.
\end{align*} 
Since the category of $U_q(\mathfrak{l}_S)$-modules is semisimple,  $\mathcal{O}_q(L_S)$ must be a cosemisimple Hopf algebra. 

\begin{defn}
A quantum flag manifold is \emph{irreducible} if the defining subset of simple roots is of the form
$$
S = \{1, \dots, r \} \setminus \{s\}
$$
where $\alpha_s$ has coefficient $1$ in the expansion of the highest root of $\mathfrak{g}$.
\end{defn}

\begin{rem}
Since the map  $\pi_S:\mathcal{O}_q(G) \to \mathcal{O}_q(L_S)$ is a surjective Hopf $\ast$-algebra morphism 
and $\mathcal{O}_q(G)$ is a compact quantum group algebra, an application of \cite[Theorem 1.6.7]{NeshveyevTuset} shows that $\mathcal{O}_q(L_S)$ is a compact quantum group algebra. We will use this fact throughout this section. 
\end{rem}

Now we are in a position to state the existence of a covariant differential calculus on $\mathcal{O}_q(G/L_S)$.

\begin{thm}\label{thm:HKClass}
Over any irreducible quantum flag manifold $\mathcal{O}_q(G/L_S)$, there exists a unique finite-dimensional left $\mathcal{O}_q(G)$-covariant differential $*$-calculus
\[
  \Omega^{\bullet}_q(G/L_S) \in \modz{\cO_q(G)}{\cO_q(G)(G/L_S)}.
\]
of classical dimension, i.e, 
\begin{align*}
  \dim \Phi\!\left(\Omega^{k}_q(G/L_S)\right) = \binom{2M}{k}, & & \text{ for all $k = 0, \dots, 2 M$},
\end{align*}
where $M$ is the complex dimension of the corresponding classical manifold.
\end{thm}

We will refer to the calculus $\Omega^{\bullet}_q(G/L_S)$ as the \emph{Heckenberger--Kolb calculus} on $\mathcal{O}_q(G/L_S)$.

We recall  the existence of a  covariant complex structure, following from the results of~\cite{HK}, \cite{HKdR}, and~\cite{MarcoConj}.

\begin{prop} \label{prop:complexstructure}
Let $\mathcal{O}_q(G/L_S)$ be an irreducible quantum flag manifold, and $\Omega^{\bullet}_q(G/L_S)$  its Heckenberger--Kolb differential $*$-calculus. Then the following hold:
\begin{enumerate}
\item $\Omega^{\bullet}_q(G/L_S)$ admits precisely two left $\mathcal{O}_q(G)$-covariant complex structures, each of which is opposite to the other, 
\item for each complex structure, $\Omega^{(1,0)}_q (G/L_S) $ and $\Omega^{(0,1)}_q (G/L_S) $ are simple objects in $\modz{\mathcal{O}_q(G)}{\mathcal{O}_q(G/L_S)}$ and are non-isomorphic to each other.
\end{enumerate}
\end{prop}

In fact, we have more. We recall that a notion of K\"ahler  structures on noncommutative differential calculus was defined and studied  in \cite{MMF3}. 

\begin{rem}
In Theorem 5.9 of \cite{MarcoConj}, Matassa shows that the Heckenberger--Kolb calculi for the quantum irreducible flag manifolds $\mathcal{O}_q (G/L_S) $ admit a K\"ahler  structure for all but finitely many values of $q$.  
\end{rem}

The following proposition will play a key role in the rest of the section.

\begin{prop}  \cite[Proposition 3.11]{HKdR}  \label{23rdoct238}
The Heckenberger-Kolb calculi on an  irreducible quantum flag manifold  is factorizable.
\end{prop}

Applying Proposition \ref{25thsep232} and the discussion in Subsection \ref{17thjuly242}  to Proposition \ref{23rdoct238}, we immediately obtain the following result:

\begin{cor} \label{23rdoct237}
If $(\Omega^{(\bullet, \bullet)}_q (G/L_S), \wedge, \partial, \overline{\partial})$ denotes the Heckenberger-Kolb calculus on an irreducible quantum flag manifold $\mathcal{O}_q (G/L_S), $  then $\Omega^{(1, 0)}_q (G/L_S) $ and $\Omega^{(0, 1)}_q (G/L_S) $ are holomorphic $\mathcal{O}_q (G/L_S)$-bimodules with respect to the given complex structure and the opposite complex structure respectively. 
\end{cor}

\subsection{Some consequences of Takeuchi's equivalence}

 We will denote $\Phi (\Omega^k_q (G/L_S)) $ and $\Phi (\Omega^{(a, b)}_q (G/L_S)) $ by the symbols $V^k$ and $V^{(a,b)}$ respectively. We will use the symbol $V$ to denote the vector space $V^1$.

\begin{lem}
  In $\modz{\cO_q(G)}{\cO_q(G/L_S)}$ we have the following:
  \begin{subequations} \label{23rdjuly242}
    \begin{align}
    &  {}  \prescript{\mathrm{co} (\mathcal{O}_q (G)) }{}{\left(\Omega_q^1(G/L_S)^{\otimes3} \right)} = 0, \label{eq:3tensLocal0} \\
   & {}  \Hom(\Omega_q^{(1,0)}(G/L_S),\ \Omega_q^2(G/L_S)) = 0, \label{23rdjuly243} \\
   & {}  \Hom(\Omega_q^{(0,1)}(G/L_S),\  \Omega_q^2(G/L_S)) = 0. \label{23rdjuly244}
  \end{align}
  \end{subequations}
\end{lem}
\begin{proof}
  Since $\Phi$ is an equivalence of monoidal categories, the statement of the lemma is equivalent to the following three equations:  
  \begin{subequations}\label{eq:HomsLocal}
  \begin{align}
    & {} \Hom_{U_q(\fl_S)}(\mathbb{C},\ V^{\otimes3}) = 0,\label{eq:3tensLocal}\\
    & {} \Hom_{U_q(\fl_S)}(V^{(1,0)},\ V^2) = 0,\label{eq:3tensLocalb}\\
    & {} \Hom_{U_q(\fl_S)}(V^{(0,1)},\ V^2) = 0. \label{eq:3tensLocalc}
  \end{align}
  \end{subequations}
  Indeed, the equivalence of \eqref{eq:3tensLocal0} and \eqref{eq:3tensLocal} follows from Corollary \ref{23rdjuly241} while the other two equivalences are straightforward. 
  
  We recall that an element $X$ of the center $Z(U_q(\fl_S))$ of $U_q(\fl_S)$ acts on any irreducible
  finite-dimensional type-1 $U_q(\fl_S)$-module~$W$ by a  central character $\chi_W \in \Hom(Z(U_q(\fl_S)),\mathbb{C})$, i.e, 
  $$ X \triangleright w = \chi_W (X) w $$
for all $w \in W$. 
  
   Now, it was noted in~\cite[\S4.4]{Fano} (also see the discussion in~\cite[\S4.5]{HolVBs}) that there is a distinguished group-like
  element $Z$ in the center of $U_q(\fl_S)$ defined as 
  \[
    \text{$Z = K^{a_1}_1 \cdots K^{a_r}_r$, where $\det (\mathcal{A}) \varpi_s = a_1 \alpha_1 + \cdots + a_r \alpha_r$},
  \]
  $\mathcal{A}$ is the Cartan matrix and $\alpha_1, \cdots \alpha_r$ are the simple roots. 
 
   In fact, from the proof of Theorem 4.12 of \cite{Fano}, we know that
 $$ \chi_{V^{(a,0)}} (Z) = q^{- a (\varpi_s, \alpha_s) \det (\mathcal{A}) }. $$
 Similarly, $\chi_{V^{(0,b)}} (Z) = q^{b (\varpi_s, \alpha_s) \det (\mathcal{A}) }$.

 Moreover, if $W_1$ and $W_2$ are $U_q(\fl_S)$-modules, then
 $$ \chi_{W_1 \otimes W_2} (Z) = \chi_{W_1} (Z) \chi_{W_2} (Z). $$
Thus, the factorizability of the calculus (Proposition \ref{23rdoct238}) implies the following equation:
\[
    \chi_{V^{(a,b)}}(Z) = q^{(b-a) (\varpi_s,\alpha_s)\det(\mathcal{A})}.
  \]
  The isomorphisms~\eqref{eq:HomsLocal} follow from these equations.
  For example, since
  \[
  V  = V^{(1,0)} \oplus V^{(0,1)},
  \]
  we see that the action of $Z$ on the summands of $V^{\otimes 3}$ are given as multiplication by the following real numbers:
  \[
    q^{3 (\varpi_s,\alpha_s)\det(\mathcal{A})},\quad
    q^{(\varpi_s,\alpha_s)\det(\mathcal{A})},\quad
    q^{- (\varpi_s,\alpha_s)\det(\mathcal{A})},\quad
    q^{-3 (\varpi_s,\alpha_s)\det(\mathcal{A})}.
  \]
  Since $(\varpi_s,\alpha_s)\neq0$ 
  and the elements of the center of~$U_q(\fl_S)$ separate irreducible
  finite-dimensional type-1 $U_q(\fl_S)$-modules, we have that~\eqref{eq:3tensLocal} holds.
  The proofs of the equivalences~\eqref{eq:3tensLocalb} and~\eqref{eq:3tensLocalc} are similar.
    \end{proof}

\begin{lem}\label{8thmay24jb4}
  Let $(\Omega^{(0,1)}_q (G/L_S), \ev_1, \coev_1)$ be a~right dual of $\Omega^{(1,0)}_q (G/L_S)$
  and $(\Omega^{(1,0)}_q (G/L_S), \ev_2, \coev_2)$ be a~right dual of $\Omega^{(0,1)}_q (G/L_S)$. 
  We have that
  \[
    \wedge\circ\coev_1(1) = \lambda \wedge\circ\coev_2(1)
  \]
  for some $\lambda \in \mathbb{C}^\times$.
\end{lem}
\begin{proof}
  Since the complex structure on $\Omega^\bullet_q(G/L_S)$ is factorizable, we have that
  \[
    \Phi(\wedge\circ\coev_1)(1),\ \Phi(\wedge\circ\coev_2)(1) \in \left(V^{(1,1)}\right)^{U_q(\fl_S)}.
  \]
  It is easy to show that $\dim \left(V^{(1,1)}\right)^{U_q(\fl_S)} = 1$, for example, see~\cite[Lemma~3.4]{Fano}.
  Therefore, there is $\lambda\in\mathbb{C}^\times$ such that
  \(
    \Phi(\wedge\circ \coev_1)(1) = \lambda \Phi(\wedge \circ \coev_2)(1)
  \)
  which proves the claim.
\end{proof}

\subsection{Metrics on the Heckenberger--Kolb calculi}

We begin by constructing a real $\mathcal{O}_q (G/L_S)$-covariant quantum symmetric (see Definition \ref{4thmay242}) metric on the space of one-forms of the Heckenberger--Kolb calculus. 

We fix  $\mathcal{O}_q (L_S) $-invariant inner products $\langle ~,~\rangle^{1,0}$ and $\langle~,~ \rangle^{0,1}$ on $V^{(1,0)}$ and $V^{(0,1)}$ respectively. We  recall the notations $((~,~)^{1,0}, \coev^{1,0})$ and $((~,~)^{0,1}, \coev^{0,1})$ introduced in Corollary \ref{29thmarch24jb4} and Remark \ref{29thmarch24jb5} respectively. Moreover, we define  
$$g^{1,0}= \coev^{1,0}(1)\quad \text{ and }\quad g^{0,1}= \coev^{0,1}(1).$$ Then by Lemma~\ref{8thmay24jb4}, it follows that there exist $\lambda \in \mathbb{C}^\times $ such that \begin{equation}
    \wedge g^{1,0} = \lambda \wedge g^{0,1}.
    \label{9thmay24jb1}
\end{equation}
Finally, we define $$g = g^{1,0} -\lambda g^{0,1}, \quad (~,~) = (~,~)^{1,0} -\lambda^{-1}  (~, ~)^{0,1}.$$

\begin{thm}\label{9thmay24jb4}
    We fix $\mathcal{O}_q(L_S)$-invariant inner-products $\langle ~,~\rangle^{1,0}$ and $\langle~,~ \rangle^{0,1}$ on $V^{(1,0)}$ and $V^{(0,1)}$ respectively. If $\lambda$ is the scalar defined in \eqref{9thmay24jb1}, then the pair $(g, (~,~))$ is a real $\mathcal{O}_q (G) $-covariant quantum symmetric metric on $\Omega^1_q(G/L_S)$.
   \end{thm}
\begin{proof}
  We start by observing that  $(g^{1,0})^\dagger=g^{1,0}$ and $(g^{0,1})^\dagger= g^{0,1}$.
Indeed, let us define
$$ \xi^{1,0}: \Omega^{(1,0)}_q (G/L_S)  \otimes_ {\mathcal{O}_q (G/L_S)}  \Omega^{(0,1)}_q (G/L_S)
\to  \mathcal{O}_q (G/L_S) $$
as in the proof of Theorem \ref{2ndapril242}. Then by \eqref{8thmay24jb2}, we have
\begin{equation} \label{12thnov241} (~,~)^{1,0}= \xi^{1,0}. \end{equation}
Now, as $g^{1,0}$ is a central $\mathcal{O}_q ( G ) $-coinvariant element, so is $ (g^{1,0})^\dagger. $
Therefore, we have a morphism
  $ \eta:  \mathcal{O}_q (G/L_S) \rightarrow \Omega^{(0,1)}_q ( G/L_S ) \otimes _ {\mathcal{O}_q (G/L_S)} \Omega^{(1,0)}_q ( G/L_S )   $
in $\modz{\mathcal{O}_q(G)}{\mathcal{O}_q(G/L_S)}$ defined by
$ \eta ( b ) = b (g^{1,0})^\dagger.$

It can be easily checked that $ ( \Omega^{(0,1)}_q ( G/L_S ), \xi^{1,0}, \eta )  $ is a right dual of $\Omega^{(1,0)}_q ( G/L_S ). $ Since  $ ( \Omega^{(0,1)}_q ( G/L_S ), ( ~ , ~ )^{1,0}, \coev^{1,0} )  $ is also a right dual of  $ \Omega^{(1,0)}_q ( G/L_S ),$ \eqref{12thnov241} implies that
$  \coev^{1,0} = \eta, $ i.e, $ g^{1,0} = (  g^{1,0}  )^\dagger. $

Similarly, it follows that  $(g^{0,1})^\dagger= g^{0,1}$. 

   Thus, 
    \begin{equation} \label{27thjuly242}
        (\wedge g^{1,0})^* = - \wedge g^{1,0} \quad \text{and} \quad (\wedge g^{0,1})^* = - \wedge g^{0,1}.
    \end{equation}
    We claim that $\lambda\in \mathbb{R}\setminus \{0\}$. Indeed, by \eqref{9thmay24jb1} and \eqref{27thjuly242}, we obtain 
    \begin{align*}
        \lambda \wedge g^{0,1}= \wedge g^{1,0} = - (\wedge g^{1,0})^* = -(\lambda \wedge g^{0,1})^*=- \overline{\lambda}(\wedge g^{0,1})^*=\overline{\lambda} \wedge g^{0,1}.
    \end{align*} 
    Since $g^{0,1} \in \Omega^{(1,0)}_q (G/L_S) \otimes_{\mathcal{O}_q (G/L_S)} \Omega^{(0,1)}_q (G/L_S) $ is not zero, the factorizability of the calculus implies that $\wedge g^{0,1} \ne 0$ and so  $\lambda= \overline{\lambda}$. 

    Thus, $g=g^{1,0} - \lambda g^{0,1}$, where $\lambda \in \mathbb{R} \setminus \{ 0 \}. $ Hence,  by Theorem \ref{2ndapril242}, we conclude   that $(g, (~,~))$ is a real $\mathcal{O}_q(G)$-covariant metric.
    Since   $\wedge g^{1,0} = \lambda \wedge g^{0,1}$, $(g, (~,~))$ is quantum symmetric. 
\end{proof}

Now we classify covariant metrics on the Heckenberger--Kolb calculi.
Let us make the following definition for the sake of brevity.
\begin{defn} \label{18thjuly241}
    We say that a metric $(g_1, (~,~)_1)$ is a scalar multiple of another metric $(g_2, (~,~)_2)$ if there is $\lambda \in \mathbb{C}^\times $ such that $g_2= \lambda g_1$.
\end{defn}

In this case, Lemma \ref{8thmay24jb1} implies that $(~,~)_2 = \frac{1}{\lambda}(~,~)_1$.

\begin{thm}\label{9thmay24jb5}
    The set of all $\mathcal{O}_q(G)$-covariant metrics on $\Omega^1_q(G/L_S)$ is a two parameter family. Moreover, there exists a unique (up to scalar) quantum symmetric $\mathcal{O}_q(G)$-covariant metric on $\Omega^1_q(G/L_S)$.
\end{thm}
\begin{proof}
 By Proposition \ref{prop:complexstructure}, $V= V^{(1,0)} \oplus V^{(0,1)}$ is a decomposition into irreducible $\mathcal{O}_q(L_S)$-comodules and $V^{(1,0)} $ is not isomorphic to $V^{(0,1)}$. Thus, the first statement follows directly from Theorem \ref{8thmay24jb5}.

 We have already proved the existence of a covariant real quantum symmetric metric $(g, (~ , ~))$ in Theorem \ref{9thmay24jb4}, where $g = g^{1,0} - \lambda g^{0,1}$, where $\lambda \in \mathbb{R} \setminus \{ 0 \}$.
 Let $ (\widetilde{g}, \widetilde{(~ , ~)}) $ be another covariant quantum symmetric metric. Then by Theorem \ref{8thmay24jb5}, $\widetilde{g} = \tau. g  $ for  an automorphism $\tau  = (\lambda_1, \lambda_2) \in \mathbb{C}^\times \times \mathbb{C}^\times $ of $V= V^{(1,0)} \oplus V^{(0,1)}.
 $ Thus, by \eqref{21stmay24}, 
 $$ \widetilde{g} = \lambda_1 g_{1,0} + \lambda_2 (- \lambda) g_{0,1}. $$
 We have
 $$ 0 = \wedge (\widetilde{g}) = \wedge (\lambda_1 g_{1,0} + \lambda_2 (- \lambda) g_{0,1}) = \lambda (\lambda_1 - \lambda_2) \wedge g_{0,1},$$ 
where we have used \eqref{9thmay24jb1}. As observed in the proof of Theorem \ref{9thmay24jb4}, $\wedge g_{0,1} \neq 0$ and therefore, $\lambda_1 = \lambda_2. $ This completes the proof of the theorem. 
\end{proof}

\begin{remark} \label{29thjuly241}
    In \cite[Section 6]{matassalevicivita}, Matassa introduces a real quantum symmetric metric for Heckenberger--Kolb calculi which coincides with the Fubini-Study metric  for classical complex projective spaces. By the uniqueness proved in Theorem \ref{9thmay24jb5}, the metric obtained in Theorem \ref{9thmay24jb4} coincides up to scalar (see Definition \ref{18thjuly241}) with the generalized Fubini-Study metric of Matassa for quantum projective spaces. 
\end{remark}

\subsection{The Levi-Civita connection}

In this subsection, we prove the existence and uniqueness of Levi-Civita connection for any real $\mathcal{O}_q (G) $-covariant metric on the space of one-forms of the Heckenberger--Kolb calculi.

 \begin{thm} \label{29thjuly242}
Let $(g, (~, ~))$ be an $\mathcal{O}_q (G) $-covariant real metric on the space of one-forms of the Heckenberger--Kolb calculi.
Then there exists a unique covariant connection $\nabla$ on $\Omega^1_q (G/L_S) $ which is torsionless and compatible with the metric $ (g, (~, ~))$. 
\end{thm}
\begin{proof}
Since $\Omega^{(1,0)}_q (G/L_S) $ and $\Omega^{(0,1)}_q (G/L_S) $
are non-isomorphic simple objects in $\modz{\mathcal{O}_q(G)}{\mathcal{O}_q(G/L_S)}$ by Proposition~\ref{prop:complexstructure}, we can apply Theorem \ref{9thmay24jb21} to get $\mathcal{O}_q (G) $-covariant Hermitian  metrics $\mathscr{H}_1$ and $\mathscr{H}_2$ on  $\Omega^{(1,0)}_q (G/L_S) $ and $\Omega^{(0,1)}_q (G/L_S) $
respectively. 

By virtue of Corollary \ref{23rdoct237}, the Chern connection $\nabla_{\Ch}$ for  $ (\Omega^{(1, 0)}_q (G/L_S), \mathscr{H}_1) $ as well as the Chern connection $\nabla_{\Ch, \op}$ (for the opposite complex structure) for the pair $(\Omega^{(0, 1)}_q (G/L_S), \mathscr{H}_2)$ exist. 

We define a  connection
\begin{equation} \label{23rdoct239}
\nabla:= \nabla_{\Ch} + \nabla_{\Ch, \op}.
\end{equation}
Then by Corollary \ref{27thjuly243}, $\nabla$ is an $\mathcal{O}_q (G) $-covariant bimodule connection. We will prove that $\nabla$ is torsionless and compatible with the metric. 

Firstly, since $\nabla$ is a bimodule connection, we can make sense of the term $\nabla g$ (see Definition \ref{4thjuly241}).

By Remark \ref{24thjuly241}, $g = \coev (1) $ for a left $\mathcal{O}_q (G)$-covariant map 
$$\coev: \mathcal{O}_q (G/L_S) \rightarrow \Omega^{1}_q (G/L_S) \otimes_{\mathcal{O}_q (G/L_S) } \Omega^{1}_q (G/L_S).$$ Therefore,  $g$ is an $\mathcal{O}_q (G)$-coinvariant element of $\Omega^{1}_q (G/L_S) \otimes_{\mathcal{O}_q (G/L_S) } \Omega^{1}_q (G/L_S)$. Hence,  $\nabla g$ is an $\mathcal{O}_q (G) $ coinvariant element of $ \Omega^1_q (G / L_S) \otimes_{\mathcal{O}_q (G/L_S) } \Omega^1_q  (G / L_S) \otimes_{\mathcal{O}_q (G/L_S) } \Omega^1_q  (G / L_S). $ Therefore, 
$\nabla g = 0$ by \eqref{eq:3tensLocal0} and so $\nabla$ is compatible with the metric $(g, (~,~))$ in the sense of Definition \ref{4thjuly241}.

Now we prove that $\nabla$ is torsionless. The torsion of $\nabla$, denoted by $T_\nabla$ is given by $T_\nabla = \wedge \circ \nabla - d$ which is clearly  $\mathcal{O}_q (G) $-covariant and  left $\mathcal{O}_q (G / L_S)$-linear. Thus, we can apply Takeuchi's equivalence (see Proposition \ref{20thdec231}) to see that $ \Phi (T_\nabla) $ is a left $U_q (\mathfrak{l}_S)  $-covariant map from $V$ to $V^2$.
However,
$$ \Hom_{U_q (\mathfrak{l}_S)} (V, V^2) \cong \Hom_{U_q (\mathfrak{l}_S)} (V^{(1,0)}, V^2) \oplus \Hom_{U_q (\mathfrak{l}_S)} (V^{(0,1)}, V^2) = 0 $$
by  \eqref{eq:3tensLocalb} and \eqref{eq:3tensLocalc}. This proves that $T_\nabla = 0$.

The uniqueness follows from the fact that there is a unique covariant connection on $\Omega^1_q (G/L_S) $ and follows along the lines of Theorem 4.5 of \cite{HolVBs}. Indeed, let $\nabla_1$ and $\nabla_2$ are two covariant connections on $\Omega^1_q (G/L_S). $ Then $\nabla_1 - \nabla_2 $ is a covariant and left $\mathcal{O}_q (G) $ linear map from $\Omega^1_q (G/L_S)$ to $\Omega^1_q (G/L_S) \otimes_{\mathcal{O}_q (G/L_S)} \Omega^1_q (G/L_S)$ so that $\Phi (\nabla_1 - \nabla_2) \in \Hom_{U_q (\mathfrak{l}_S)} (V, V \otimes V). $ 

Since  $\Omega^{(1,0)}$ and $ \Omega^{(0,1)}$ are simple, the proof of Theorem 4.5 of \cite{HolVBs} shows that   $\Hom_{U_q (\mathfrak{l}_S)}(V^{(1,0)}, V \otimes V^{(1,0)})$ and  $\Hom_{U_q (\mathfrak{l}_S)}(V^{(0,1)}, V \otimes V^{(0,1)})$ are zero. By a similar computation,  $\Hom_{U_q (\mathfrak{l}_S)}(V^{(1,0)}, V \otimes V^{(0,1)})$ and $\Hom_{U_q (\mathfrak{l}_S)}(V^{(0,1)}, V \otimes V^{(1,0)})$ are also zero. Therefore,   $\Hom_{U_q (\mathfrak{l}_S)}(V, V \otimes V)= 0 $ implying that $\nabla_1 - \nabla_2 = 0$. 
\end{proof}

\begin{rem}
By virtue of Remark \ref{29thjuly241}, Theorem \ref{29thjuly242} proves the existence and uniqueness of Levi-Civita connection for the  generalized Fubini-Study metric constructed by Matassa in \cite{matassalevicivita}. Thus, Theorem \ref{29thjuly242} generalizes   \cite[Theorem 8.4]{matassalevicivita}.
\end{rem}

Since $\nabla$ is a Levi-Civita connection, it is automatically cotorsion free. We end the section by giving an independent  short proof of this fact.

\begin{lem}
  The connection $\nabla$ defined in \eqref{23rdoct239} is cotorsion free.
\end{lem}
\begin{proof}
  The cotorsion of $\nabla$ is 
  defined as
  \[ \coT_{\nabla}:=
  \left(\exd\otimes_{\mathcal{O}_q (G/L_S)} \id - (\wedge\otimes_{\mathcal{O}_q (G/L_S)} \id)\circ(\id\otimes_{\mathcal{O}_q (G/L_S)} \nabla) \right) (g).
       \]
  We observe that the map 
  $\exd\otimes_{\mathcal{O}_q (G/L_S)} \id - (\wedge\otimes_{\mathcal{O}_q (G/L_S)} \id)\circ(\id\otimes_{\mathcal{O}_q (G/L_S)} \nabla)$ from 
  $\Omega_q^1(G/L_S)\otimes_{\mathcal{O}_q (G/L_S)} \Omega_q^1(G/L_S) $ to $\Omega^2_q(G/L_S)\otimes_{\mathcal{O}_q (G/L_S)} \Omega^1_q(G/L_S)$
   is left $\cO_q(G)$ -covariant.
   
  Since the element $g\in {}^{\co(\cO_q(G))}\left(\Omega^1_q(G/L_S)\otimes_{\mathcal{O}_q (G/L_S)}\Omega_q(G/L_S)\right)$, we
  have that $\coT_\nabla \in {}^{\co(\cO_q(G))}(\Omega^2_q(G/L_S)\otimes_{\mathcal{O}_q (G/L_S)} \Omega^1_q(G/L_S)) = 0$ by a combination of~\eqref{23rdjuly243} and~\eqref{23rdjuly244}.
\end{proof}

\section{The covariance of the Chern connection} \label{1staugust241}

Suppose $ (\Omega^{(\bullet, \bullet)}, \wedge, \partial, \overline{\partial}) $ is an $A$-covariant complex structure over an $A$-comodule $\ast$-algebra $B$ and $(\mathcal{E}, \overline{\partial}_{\mathcal{E}})$ is an $A$-covariant holomorphic $B$-bimodule in the sense of Definition \ref{11thdec231}. Our goal in this section is to prove that if $\mathscr{H}$ is an $A$-covariant Hermitian metric on $\mathcal{E}$, then the Chern connection for the pair $(\mathcal{E}, \mathscr{H})$ is $A$-covariant. This was stated as Theorem \ref{23rdoct231} and was used in Section \ref{4thaugust243} and Section \ref{4thaugust244}. 

Theorem \ref{23rdoct231} was known, but since the proof is not available in the literature, we decided to include it for the sake of completeness. After setting up some notations and proving some preparatory results,   we deduce a formula for the $(1, 0)$-part of the Chern connection in the second subsection, while in the third subsection, we use this formula to verify the covariance of the Chern connection.  

Throughout this section, we will be using the notions of $\partial$-connections, $\overline{\partial}$-connections and holomorphic bimodules from Subsection \ref{4thaugust241}. Moreover, the compatibility of a connection with an Hermitian metric and the definition of the Chern connection have been defined in Subsection \ref{4thaugust242}. 

\subsection{Some preparatory results}
Let $ (\Omega^{\bullet}, \wedge, d)$ be a differential calculus on an algebra $B$. 
If $\mathcal{E}$ is  finitely generated and projective as a left $B$-module and $\{ e^i, e_i: i = 1, \cdots n \}$ is a dual basis for $\mathcal{E}$ as in \eqref{27thnov232}, then for a connection $\nabla$ on $\mathcal{E}$, there exist elements $\Gamma^i_j$ in $\Omega^1 (B) $ called the Christoffel symbols of the connection $\nabla$ such that 
$$\nabla (e^i) = - \sum^n_{j = 1} \Gamma^i_j \otimes_B e^j.$$

We have the following result:  

\begin{prop} \label{beggsmajidchriistoffel} (\cite[Proposition 3.23]{BeggsMajid:Leabh})
Given a left connection $\nabla $ on a $B$-bimodule $\mathcal{E}$ which is finitely generated and projective as a left $B$-module, consider the matrices 
\begin{align}\label{P Matrix}
    \Gamma=\left (\Gamma^i_j\right)_{i,j} ~ \text{and} ~ P = \left (\ev(e^i\otimes_B e_j)\right)_{i,j}.
\end{align}
Then $\Gamma P= \Gamma $ and $\Gamma = P\Gamma -(dP) P$. 

Conversely, a choice of Christoffel symbols satisfying these conditions define a left connection on $\mathcal{E}$.
\end{prop}

Now if $ B $ is a $\ast$-algebra and $ (\Omega^{(\bullet, \bullet)}, \wedge, \partial, \overline{\partial}) $ is a complex structure on $B$, then we have an analogous result for a $\partial$-connection on $\mathcal{E}$.

\begin{rem}
Suppose $\{ e^i, e_i: i = 1, \cdots n \} $ be a dual basis for $\mathcal{E}$ and let  $\partial_{\mathcal{E}}$ be a $\partial$-connection on $\mathcal{E}$ such that
$\partial_{\mathcal{E}}(e^i)= -\sum_j s^i_j\otimes_B e^j $
for some elements $s^i_j$ in $\Omega^{(1, 0)}$. If $s$ denotes the matrix whose $(i, j)$-th entry is $s^i_j$
 and $P$ is the matrix defined  in \eqref{P Matrix}, then by a verbatim adaptation of the proof of Proposition \ref{beggsmajidchriistoffel}, we can deduce  the following two equations:
\begin{align}\label{sp=s}
    sP=s ~ \text{and} ~     s=Ps-(\partial P)P.
\end{align}
\end{rem}

Now we introduce some notations and collect a few identities for  Hermitian metrics. So let $\{ e^i, e_i: i = 1, \cdots n  \}$ be a dual basis of a $B$-bimodule $\mathcal{E}$ as above.   Then, given an Hermitian metric $\mathscr{H}$ on $\mathcal{E}$,  there exist elements $h^{ij}, h_{ij}$ in $B$ for all $i,j \in \{ 1, \cdots, n \}$, such that
\begin{equation} \label{3rddec234}
 {h^{ij}}=\langle e^i, \overline{e^j}\rangle, ~ \mathscr{H}^{-1}(e_i) =\sum_j \overline{ h_{ij}e^j}.
\end{equation}
 
In the next proposition, we collect some useful identities which connect $h^{ij}, h_{ij}$ and the matrix $P$ defined in \eqref{P Matrix}. Here and elsewhere, for a matrix $X=(x_{ij})_{i,j} \in M_n(B)$, $X^\dagger$ will denote the $B$-valued matrix  $ X^{\dagger}=(x_{ji}^{\ast})_{i,j}$.

\begin{prop} (\cite[Proposition 4.2]{beggs2011compatible}) \label{metric equations}
If $\mathscr{H}$ is an~Hermitian metric on a~$B$-bimodule $\mathcal{E}$ which is finitely generated and projective as a~left $B$-module. With $h^{ij}$ and $h_{ij}$ as in~\eqref{3rddec234}, we will let $h$ and $\tilde{h}$ denote the matrices whose $(i,j)$-th entries are $h^{ij}$ and $h_{ij}$ respectively.
Then the following equations hold:
\begin{gather*}
  \mathscr{H}(\overline{e^i}) =\sum_j e_j h^{ji},\ h \tilde{h}=P,\ \tilde{h} h = P^{\dagger},\\
  (h)^{\dagger} =h,\ (\tilde{h})^{\dagger}=\tilde{h},\ \tilde{h}P= \tilde{h},\ Ph=h.
\end{gather*}
\end{prop}

As a~consequence, we have the following proposition:

\begin{prop} \label{g tilde delbar P g=0} 
Let $(\Omega^{(\bullet, \bullet)}, \wedge, \partial, \overline{\partial})$ be a~complex structure on a~$\ast$-algebra $B$ and $\mathcal{E}$ be a~$B$-bimodule which is finitely generated and projective as a~left $B$-module and $\mathscr{H}$ be an~Hermitian metric on~$\mathcal{E}$. If $P$ denotes the matrix defined in~\eqref{P Matrix}, then with the notations as above, we have
\begin{align*}
    \tilde{h}\overline{\partial}(P) h = 0.
\end{align*}
\end{prop}
\begin{proof}
  Using the fact that $\overline{\partial}$ satisfies the Leibniz rule and the equation $\tilde{h} P = \tilde{h}$ and $P h = h$ from Proposition~\ref{metric equations}, we obtain
$$   \tilde{h}\overline{\partial}(P)h =\overline{\partial}(\tilde{h}P) h -\overline{\partial}(\tilde{h})P h = \overline{\partial}(\tilde{h}) h -\overline{\partial}(\tilde{h}) h =0. $$
\end{proof}

Let $\{ e^i, e_i: i = 1, \cdots, n \} $ be a dual basis of a holomorphic $B$-bimodule $\mathcal{E}$ as above. Let us write 
$$ \overline{\partial}_{\mathcal{E}} (e^i) = - \sum^n_{j = 1} (\Gamma_{-})^i_j \otimes_B e^j $$
for some elements $ (\Gamma_{-})^i_j$ in $\Omega^{(0,1)}$.
Then from the proof of \cite[Proposition 4.2]{BeggsMajidChern}, it follows that the $(1, 0)$-part of the Chern connection for a pair $(\mathcal{E}, \mathscr{H})$ is given as
$$ \nabla^{(1, 0)} (e^i) = - \sum^n_{j = 1} (\Gamma_{+})^i_j \otimes_B e^j, $$
where the matrix $ \Gamma_{+}: = ((\Gamma_{+})^i_j)_{i,j} $ is determined by the equation
\begin{equation} \label{4thdec232}
 - \Gamma_+ = \partial h. \tilde{h} + h (\Gamma_{-})^{\dagger} \tilde{h},
 \end{equation}
and the matrices $h$ and $\tilde{h}$ are defined in Proposition \ref{metric equations}.

\subsection{The $(1, 0)$-part of the Chern connection}

Our construction of the $(1, 0)$-part of the Chern connection is inspired by and closely related to the formula derived in \cite[Proposition 7.3]{OSV}. However, our novelty lies in the usage of the language of bar-categories which makes the construction more transparent.

We recall \cite[Proposition 3.32]{BeggsMajid:Leabh} for the anti-holomorphic part of the complex structure as the first step of our construction.

\begin{prop} \label{del tilde lemma}
For a left $\overline{\partial}$-connection $\overline{\partial}_{\mathcal{E}}$ on a $B$-bimodule $\mathcal{E}$ which is finitely generated and projective as a left $B$-module, there is a unique right $\overline{\partial}$-connection $\widetilde{\overline{\partial}}_{\mathcal{E}}: \prescript{}{B}{\Hom(\mathcal{E},B)} \to  \prescript{}{B}{\Hom(\mathcal{E},B)} \otimes_B \Omega^{(0,1)} $ such that 
\begin{align}\label{uniqueness of delbar}
    \overline{\partial}\circ \ev (e \otimes_B f) = (\id\otimes_B\ev)(\overline{\partial}_{\mathcal{E}}(e) \otimes_B f) + (\ev\otimes_B\id)(e\otimes_B \widetilde{\overline{\partial}}_{\mathcal{E}}(f))
\end{align}
for all $e\in \mathcal{E}$ and $f\in \prescript{}{B}{\Hom(\mathcal{E},B)}$.
If $\{ e^i, e_i: i = 1, \cdots n \}$ is a dual basis of $\mathcal{E}$, then
    \begin{align}\label{del tilde_E}
    \widetilde{\overline{\partial}}_{\mathcal{E}}(f)=\sum_i \bigl(e_i\otimes_B\overline{\partial}(\ev(e^i\otimes_Bf))-e_i\otimes_B(\id\otimes_B \ev)(\overline{\partial}_{\mathcal{E}}(e^i)\otimes_Bf)\bigr).
\end{align}
\end{prop}

\begin{cor} \label{rightdelbar on Ebar} 
Suppose $\overline{\partial}_{\mathcal{E}}$ is a left $\overline{\partial}$-connection on a $B$-bimodule $\mathcal{E}$ as in Proposition \ref{del tilde lemma},  $\widetilde{\overline{\partial}}_{\mathcal{E}}$ is the right $\overline{\partial}$-connection on $\prescript{}{B}{\Hom(\mathcal{E},B)}$ defined in \eqref{del tilde_E} and $\mathscr{H}$ is an Hermitian metric on $\mathcal{E}$ as in Definition \ref{4thdec231}. Then the map 
${\widetilde{\overline{\partial}}}_{\mathcal{E},\mathscr{H}}:\overline{\mathcal{E}}\to\overline{\mathcal{E}}\otimes_B \Omega^{(0,1)} $
    defined as 
\begin{align}\label{tilde tilde del E}
        {\widetilde{\overline{\partial}}}_{\mathcal{E},\mathscr{H}}(\overline{e})= (\mathscr{H}^{-1}\otimes_B\id)\widetilde{\overline{\partial}}_{\mathcal{E}}(\mathscr{H}(\overline{e})) 
\end{align}
 is a right $\overline{\partial}$-connection on $\overline{\mathcal{E}}$. 
\end{cor}
\begin{proof}
     Since $\mathscr{H}$ is a $B$-bimodule isomorphism, this follows from the fact that $\widetilde{\overline{\partial}}_{\mathcal{E}}$ is a right $\overline{\partial}$-connection on $\overline{\mathcal{E}}$.   
\end{proof}

\begin{defn} \label{right del on E}
With the notations of Corollary \ref{rightdelbar on Ebar}, assume that 
   ${\widetilde{\overline{\partial}}}_{\mathcal{E},\mathscr{H}}(\overline{e})=\sum_j\overline{f_j}\otimes_B \omega_j $. 
Then we define the map 
\begin{align} \label{(1,0)part}
    \widehat{\nabla}:\mathcal{E}\to \Omega^{(1,0)}\otimes_B \mathcal{E} ~ \text{by} ~ \widehat{\nabla}(e)= \sum_j \omega_j^{\ast}\otimes_B f_j.
\end{align} 
\end{defn}

\begin{rem} \label{29thnov231}
Proposition \ref{25thnov235} implies that $\widehat{\nabla}$ is well-defined. Moreover, it is easy to see that  $\widehat{\nabla}$ is a left $\partial$-connection on $\mathcal{E}$.
\end{rem}

Now we state and prove the main result of this subsection.

\begin{thm} \label{1,0partofchernconnection}
Let $ (\Omega^{\bullet}, \wedge, \partial, \overline{\partial}) $ be a complex structure on a $\ast$-algebra $B$ and $\mathscr{H}$ be an Hermitian metric on a holomorphic $B$-bimodule $ (\mathcal{E}, \overline{\partial}_{\mathcal{E}})$. If $\mathcal{E}$ is finitely generated and projective as a left $B$-module, then the connection $\widehat{\nabla}$  
 of \eqref{(1,0)part} coincides with  the $(1,0)$ part of the Chern connection on $\mathcal{E}$, i.e., 
 $$\widehat{\nabla}= (\pi^{(1,0)}\otimes_B \id) \nabla_{\Ch}.$$ 
\end{thm}
\begin{proof}    
 Let $\{ e^i, e_i \}$ be a dual basis for the $B$-bimodule $\mathcal{E}$.
We assume that
\begin{equation} \label{21stnov232}
\overline{\partial}_{\mathcal{E}}(e^i)=-\sum_j(\Gamma_{-})^i_j\otimes_B e^j ~ \text{and} ~\widehat{\nabla}(e^i)= -\sum_js^i_j\otimes_B e^j.
\end{equation}

Consider the matrices $s := (s^i_j)_{ij} $ and $\Gamma_{-}: = ((\Gamma_{-})^i_j)_{ij} $. We use the notations $s^{\dagger} := ((s^j_i)^{\ast})_{ij} $ and $\Gamma_{-}^{\dagger}: = (((\Gamma_{-})^j_i)^{\ast})_{ij} $. 
Then by virtue of \eqref{4thdec232}, it suffices to show   that 
\begin{align*}
    -s=\partial h\cdot \tilde{h } + h (\Gamma_{-})^{\dagger} \tilde{h}.
\end{align*}
Now, by the definition of $\widehat{\nabla}$ and \eqref{21stnov232}, we get   $\widetilde{\overline{\partial}}_{\mathcal{E}, \mathscr{H}}(\overline{e^i})=- \sum_j \overline{e^j}\otimes_B (s^i_j)^{\ast}$.  Therefore, from \eqref{tilde tilde del E}, we have
\begin{equation} \label{21stnov231}
    \widetilde{\overline{\partial}}_{\mathcal{E}}(\mathscr{H}(\overline{e^i}))= (\mathscr{H}\otimes_B\id) \widetilde{\overline{\partial}}_{\mathcal{E}, \mathscr{H}}(\overline{e^i}) = - \sum_j \mathscr{H}(\overline{e^j})\otimes_B (s^i_j)^{\ast}.
\end{equation}
Using the formula for $\mathscr{H} (\overline{e^i}) $ and $\mathscr{H} (\overline{e^j}) $ from Proposition \ref{metric equations} in \eqref{21stnov231}, we have
\begin{align*}
   \sum_j  \widetilde{\overline{\partial}}_{\mathcal{E}}(e_jh^{ji})= - \sum_{j,k} e_k h^{kj}\otimes_B(s^i_j)^{\ast}.
\end{align*}
Now using the fact that $\widetilde{\overline{\partial}}_{\mathcal{E}}$ is a right $\overline{\partial } $-connection, we get,
\begin{align*}
    - \sum_{j,k} e_k h^{kj}\otimes_B(s^i_j)^{\ast}= \sum_j \bigl(\widetilde{\overline{\partial}}_{\mathcal{E}}(e_j)h^{ji}+e_j\otimes_B \overline{\partial}(h^{ji}) \bigr).
\end{align*}
Using \eqref{del tilde_E}, this implies
\begin{align*}
    &- \sum_{j,k} e_k h^{kj}\otimes_B(s^i_j)^{\ast}\\
    &= \sum_{j,l} \bigl(e_l \otimes_B \overline{\partial} (\ev (e^l \otimes_B e_j)) h^{ji} - e_l \otimes_B (\id \otimes_B \ev) (\overline{\partial}_{\mathcal{E}} (e^l) \otimes_B e_j) h^{ji} \bigr) + \sum_j e_j \otimes_B \overline{\partial} (h^{ji})\\
    &= \sum_{j,l} \bigl(e_l\otimes_B \overline{\partial}(\ev(e^l\otimes_Be_j))h^{ji}  + \sum_k e_l\otimes_B (\Gamma_-)^l_k\ev(e^k\otimes_Be_j) h^{ji} \bigr) + \sum_j e_j\otimes_B\overline{\partial}(h^{ji})
\end{align*}
by the first equation of \eqref{21stnov232}.

We recall that $P$ denotes the matrix defined in  \eqref{P Matrix}. Tensoring the above equation on the left by $e^t$, we get
\begin{align*}
    &-e^t\otimes_B \sum_{j,k} e_k h^{kj}\otimes_B(s^i_j)^{\ast}\\
    &= e^t\otimes_B \sum_{j,l} e_l\otimes_B \overline{\partial}(\ev(e^l\otimes_Be_j))h^{ji} +e^t\otimes_B \sum_{j,k,l} e_l\otimes_B (\Gamma_-)^l_k\ev(e^k\otimes_Be_j) h^{ji}\\
    &+e^t\otimes_B \sum_j e_j\otimes_B \overline{\partial}(h^{ji}).
\end{align*}
for all $t$. Now we apply $\ev\otimes_B \id$ to the above equation to get
\begin{align*}
    -Phs^{\dagger}= P\overline{\partial}(P)h+P (\Gamma_-)Ph+P \overline{\partial}(h)
\end{align*}
We apply  the equation $Ph = h$ of Proposition \ref{metric equations} to obtain 
\begin{align*}
    - hs^{\dagger}= P\overline{\partial}(P)h+P (\Gamma_-)h+P \overline{\partial}(h).
\end{align*}
Multiplying by $\tilde{h}$ on the left  and using the relations $\tilde{h} h = P^\dagger $ and $\tilde{h} P = \tilde{h} $ of Proposition \ref{metric equations}, we get
\begin{align*}
    -P^{\dagger}s^{\dagger}= \tilde{h}\overline{\partial}(P)h+\tilde{h} (\Gamma_-)h+\tilde{h} \overline{\partial}(h).
\end{align*}
Therefore, from Proposition \ref{g tilde delbar P g=0}, we get 
\begin{equation} \label{29thnov232}
    -(sP)^{\dagger}= \tilde{h} (\Gamma_-)h+\tilde{h} \overline{\partial}(h).
\end{equation}
But as $\widehat{\nabla}$ is a left $\partial$-connection by Remark \ref{29thnov231}, we have  $sP=s$ from \eqref{sp=s}. Using this,  and taking $\dagger$ on \eqref{29thnov232}, we have
\begin{align*}
    -s= {h} (\Gamma_-)^{\dagger}\tilde{h}+ {\partial}(h)\tilde{h},
\end{align*}
where we have used the equations $h^\dagger = h$ and $ (\tilde{h})^\dagger = \tilde{h}$ from Proposition \ref{metric equations}.
This completes our proof. 
\end{proof}

\subsection{Proof of the covariance of the Chern connection}

In this subsection, we use Theorem \ref{1,0partofchernconnection} to prove that the Chern connection on a covariant holomorphic $B$-bimodule for a covariant Hermitian metric is covariant. Throughout the section, we will use the fact that the category $\qMod{A}{B}{}{B}$ is a bar-category (see Example \ref{2nddec231}). 

We begin by proving the following lemma:

\begin{lem} \label{tilde_partial_E_H_covariant}
Let $ (\mathcal{E}, \overline{\partial}_{\mathcal{E}})$ be a left $A$-covariant holomorphic $B$-bimodule  in $\qMod{A}{B}{}{B}$ which is finitely generated and projective as a left $B $-module. If  $\mathscr{H}$ is a left $A$-covariant Hermitian metric on $\mathcal{E}$, then the right $\overline{\partial}$-connections $\widetilde{\overline{\partial}}_{\mathcal{E}}$ and $\widetilde{\overline{\partial}}_{\mathcal{E},\mathscr{H}}$ defined by \eqref{del tilde_E} and \eqref{(1,0)part} respectively are left covariant.
\end{lem}
\begin{proof} Let $f $ belongs to $\prescript{}{B}{\Hom(\mathcal{E},B)}$. Then \eqref{del tilde_E} implies that 
\begin{align*}
    \widetilde{\overline{\partial}}_{\mathcal{E}}(f)
   & = \Bigl(\id \otimes_B(\overline{\partial}\circ \ev)- (\id\otimes_B (\id\otimes_B \ev)\circ (\overline{\partial}_{\mathcal{E}} \otimes_B \id)) \Bigr) \Bigl(\coev\otimes_B \id \Bigr) (1_B \otimes_B f).
\end{align*}
Now as the maps $\overline{\partial}_\mathcal{E}, \overline{\partial}$ are $A$-covariant,
 $\widetilde{\overline{\partial}}_{\mathcal{E}}$ is also $A$-covariant. 

The metric $\mathscr{H}$ on $\mathcal{E}$ is covariant. Hence, the covariance of $\widetilde{\overline{\partial}}_{\mathcal{E},\mathscr{H}}$ follows from the definition of $\widetilde{\overline{\partial}}_{\mathcal{E},H}$ and the covariance of the map $\widetilde{\overline{\partial}}_{\mathcal{E}}$. 
\end{proof}

Finally, we prove  the covariance of the Chern connection.

\begin{thm} \label{30thjuly241} (Theorem \ref{23rdoct231})
Let $ (\Omega^{(\bullet, \bullet)}, \wedge, \partial, \overline{\partial}) $ be an $A$-covariant  complex structure on a $\ast$-algebra $B$. If $ (\mathcal{E}, \overline{\partial}_{\mathcal{E}})$ is a covariant holomorphic $B$-bimodule equipped with an $A$-covariant Hermitian metric $\mathscr{H}$ with $\mathcal{E}$ finitely generated and projective as a left $B$-module, then the Chern connection
$\nabla_{\Ch}$ is a covariant connection.
\end{thm}
\begin{proof} Let  $ \widehat{\nabla} $ be the $\partial$-connection introduced in equation \eqref{(1,0)part}. By a combination of Theorem \ref{chern} and Theorem \ref{1,0partofchernconnection}, we know that
 \begin{equation} \label{24thnov23n1} \nabla_{\Ch} = \overline{\partial}_{\mathcal{E}} + \widehat{\nabla}. \end{equation}
 Since the holomorphic structure on $\mathcal{E}$ is $A$-covariant, the map $\overline{\partial}_{\mathcal{E}}$  is $A$-covariant by definition. 
Therefore, by virtue of \eqref{24thnov23n1}, we are left to prove that $\widehat{\nabla}$ is covariant.

Let $e $ be an element of $\mathcal{E}$. Since  $\Omega^{(0,1)}= B \overline{\partial} B$, there exist elements  $f_i \in \mathcal{E}$ and $a_i \in B $ such that 
\begin{equation} \label{12thdec232}
\widetilde{\overline{\partial}}_{\mathcal{E},\mathscr{H}} (\overline{e})= \sum_i \overline{f}_i\otimes_B \overline{\partial}(a_i).
\end{equation}
Let $    \prescript{\mathcal{E}}{}{\delta}(e)= e_{(-1)} \otimes e_{(0)}$ denote the coaction of $A$ on $\mathcal{E}$. Then by using \eqref{25thnov232}, we have $\prescript{\overline{\mathcal{E}}}{}{\delta}(\overline{e})= e_{(-1)}^{\ast} \otimes \overline{e_{(0)}}$.

Now by Lemma \ref{tilde_partial_E_H_covariant}, $\widetilde{\overline{\partial}}_{\mathcal{E},\mathscr{H}}$ is $A$-covariant. So we have
\begin{align}\label{25thnov233}
    \prescript{\overline{\mathcal{E}}\otimes_B\Omega^{(0,1) }}{ }{\delta}(\widetilde{\overline{\partial}}_{\mathcal{E},\mathscr{H}} (\overline{e}))
     = (\id\otimes \widetilde{\overline{\partial}}_{\mathcal{E},\mathscr{H}}) \prescript{\overline{\mathcal{E}}}{}{\delta}(\overline{e})    
     = e^{\ast}_{(-1)}\otimes\widetilde{\overline{\partial}}_{\mathcal{E},\mathscr{H}} {(\overline{e_{(0)}})}.
\end{align}
Let 
 \begin{equation} \label{12thdec231}  
    \widetilde{\overline{\partial}}_{\mathcal{E},\mathscr{H}}{(\overline{e_{(0)}})}= \sum_j \overline{ x_j}\otimes_B\omega_j
    \end{equation}
  for some $x_j \in \mathcal{E} $ and $\omega_j \in \Omega^{(0, 1)}$.  Then from \eqref{25thnov233} and the left $A$-covariance of the map $\overline{\partial}$, we have 
\begin{align}{\label{25thnov231}}
    \sum_i (f_i)_{(-1)}^{\ast}(a_i)_{(1)} \otimes \overline{(f_i)_{(0)}}\otimes_B \overline{\partial}((a_i)_{(2)})
    =\sum_j e^{\ast}_{(-1)}\otimes \overline{ x_j}\otimes_B\omega_j,
\end{align}
where we have used the Sweedler notations
$    \prescript{\mathcal{E}}{}{\delta}(f_i)= (f_i)_{(-1)} \otimes (f_i)_{(0)},   $
$    \Delta (a_i) = (a_i)_{(1)}\otimes (a_i)_{(2)}$ and
using \eqref{25thnov232}, we have written $\prescript{\overline{\mathcal{E}}}{}{\delta}(\overline{f_i})= (f_i)_{(-1)}^{\ast} \otimes \overline{(f_i)_{(0)}}$.

At this stage, from Remark \ref{11thdec23n1}, we recall that the map $\star_{\Omega^1}: \Omega^{(0,1)} \rightarrow \overline{\Omega^{(1,0)}} $ is an isomorphism and $\Upsilon$ is the natural isomorphism in Definition \ref{15thjuly241}. We apply the linear map 
\begin{align*}
    \id\otimes\Upsilon_{\Omega^{(1,0)},\mathcal{E}}^{-1}(\id\otimes_B \star_{\Omega^1}) : A\otimes\overline{\mathcal{E}}\otimes_B\Omega^{(0,1)} \to  A \otimes \overline{\Omega^{(1,0)} \otimes_B {\mathcal{E}}}
\end{align*} 
to \eqref{25thnov231} and use the third assertion of Lemma \ref{12thdec233}  to get
\begin{align*}
     \sum_j (f_i)_{(-1)}^{\ast} (a_i)_{(1)} \otimes  \overline{\partial((a_i)_{(2)}^{\ast})\otimes_B (f_i)_{(0)}}= \sum_j  e_{(-1)}^{\ast}\otimes \overline{\omega_j^{\ast} \otimes_B x_j} .
\end{align*}
By Lemma \ref{24thnov231}, this implies the following equation: 
\begin{align*} 
   \sum_i (a_i)_{(1)}^{\ast}(f_i)_{(-1)}  \otimes \partial((a_i)_{(2)}^{\ast})\otimes_B (f_i)_{(0)}= \sum_j  e_{(-1)}\otimes  \omega_j^{\ast} \otimes_B x_j.
\end{align*}
 Therefore,
 \begin{align*}
     (\id\otimes \widehat{\nabla}) \prescript{{\mathcal{E}}}{}{\delta}({e})
&{}= e_{(-1)} \otimes \widehat{\nabla} (e_{(0)})\\     
&{}= \sum_j  e_{(-1)}\otimes  \omega_j^{\ast} \otimes_B x_j\quad\text{(by~\eqref{12thdec231} and Definition~\ref{right del on E})}\\
&{}= \sum_i (a_i)_{(1)}^{\ast}(f_i)_{(-1)}  \otimes \partial((a_i)_{(2)}^{\ast})\otimes_B (f_i)_{(0)}\\
&{}= \prescript{\Omega^{(1,0)}\otimes_B \mathcal{E}}{ }{\delta} (\sum_j \partial (a^*_j) \otimes_B f_j)\\
&{}= \prescript{\Omega^{(1,0)}\otimes_B \mathcal{E}}{ }{\delta}(\widehat{\nabla}({e}))
 \end{align*}
by \eqref{12thdec232} and Definition \ref{right del on E}.
 Thus, $\widehat{\nabla}$ is a left covariant.
\end{proof}

\section{Appendix} \label{10thaugust244}

In this appendix, we collect the proof of some results which have used in the article.

\subsection{Compact quantum group algebras} \label{3rdjuly242}

Throughout this article, we assume familiarity with the theory of compact quantum groups. For details on this theory, we refer to \cite{woroleshouches} and \cite{NeshveyevTuset}.  

If a Hopf $\ast$-algebra $H$ is the span of the matrix coefficients of unitary irreducible corepresentations of a compact quantum group $\widetilde{H}$, then we say that $H$ is a compact quantum group algebra.  

Let $H$ be a compact quantum group algebra and let $V$ be an object of the category $\lmod{H}{}$ consisting of finite dimensional left $H$-comodules. Then $V$ is a finite dimensional corepresentation of the compact quantum group $\widetilde{H}$. It is well-known that $V$ is unitarizable, i.e, there exists an inner product  $ \left\langle ~  , ~ \right\rangle $ (conjugate linear on the right)  on $V$ so that the following equation is satisfied
\begin{equation} \label{28thmarch24n2}
v_{(-1)} w^*_{(- 1)} \left\langle v_{(0)}, w_{(0)} \right\rangle = \left\langle v, w \right\rangle 1_H
\end{equation}
for all $v, w$ in $V$. 

If an inner product $ \left\langle ~ , ~ \right\rangle $ on a left $H$-comodule $V$ satisfies \eqref{28thmarch24n2}, then we say that  $ \left\langle ~ , ~ \right\rangle $ is $H$-invariant.

Let us record the following simple lemma, where, for an object $ (V, \prescript{V}{}{\delta}) $ in $\lmod{H}{}$, $\overline{V}$ will denote the conjugate vector space equipped with the left comodule coaction defined in \eqref{25thnov232}.

\begin{lem}\label{28feb241}
Let $H$ be a compact quantum group algebra and let $(V,  \prescript{V}{}{\delta})$ be a finite dimensional left $H$-comodule. If $\langle ~,~ \rangle$ is an $H$-invariant inner product on $V$ in the sense of \eqref{28thmarch24n2}, then 
\begin{equation*}
    \psi: \overline{V} \to \Hom(V, \mathbb{C}), \quad \text{defined by } \psi(\overline{v})(w)=\langle w,v\rangle. 
\end{equation*} 
is an isomorphism in the category $\lmod{H}{}$. Moreover, $\overline{V}$ is a right dual of $V$ in $\lmod{H}{}$ via the morphisms:
    \begin{align}{\label{30thmarch241}}
        &V\otimes \overline{V}\to \mathbb{C}; \quad v\otimes \overline{w} \mapsto \langle v, w\rangle \quad \text{and}
        \quad \mathbb{C} \to \overline{V}\otimes V;  \quad 1\mapsto \sum_{i=1}^{\dim(V)}(\overline{v_i}\otimes v_i),
    \end{align} 
    where $\{v_i\}_{i=1}^{\dim(V)}$ is an orthonormal basis of $V$.
\end{lem}

\subsection{Regarding Takeuchi's equivalence} \label{3rdjuly241}

We will continue to use the notations of Subsection \ref{19thdec235}. In particular, for a quantum homogeneous space $B = A^{\co(H)}$, $\Phi$ and $\Psi$ will denote the functors between $\qMod{A}{B}{}{B}$ and $\qMod{H}{}{}{B}$ introduced in \eqref{2ndjuly241} and \eqref{2ndjuly242}. Takeuchi's original result was in the following form:

\begin{thm} (Theorem 1, \cite{Tak}) \label{takeuchi} 
 Let $B = A^{\co(H)}$ be a quantum homogeneous space as in Definition \ref{ Quantum homogeneous space }. Then an adjoint equivalence of the categories $ \qMod{A}{B}{}{B} $ and $ \qMod{H}{}{}{B}  $ is given by $\Phi, \Psi$ and the unit natural isomorphism 
 \begin{equation} \label{6thjune242}
 U: M \rightarrow \Psi \circ \Phi (M), ~ m \rightarrow m_{(-1)} \otimes [ m_{(0)} ].
 \end{equation} 
\end{thm}

In Corollary 2.7 of \cite{MMF2}, it was observed that the inverse of $U$ is given by
\begin{equation} \label{6thjune243}
 U^{-1} \left(\sum_i a_i \otimes [ m_i ] \right) = \sum_i a_i S ((m_i)_{(- 1)}) (m_i)_{(0)}.  
\end{equation}

The aim of this subsection is to state Proposition \ref{20thdec231}. Though this result is well-known to the experts, we include its proof since it has been used in the article more than once. We will need to introduce some more notations. 

Consider the full subcategory $ \Modz{A}{B} $  of $\qMod{A}{B}{}{B} $ consisting of objects $M$ which satisfies $ M B^+ = B^+ M. $ Then, under Takeuchi's equivalence (Theorem \ref{takeuchi}), the   corresponding full subcategory $\Modz{H}{}$ of $ \qMod{H}{}{}{B}$   is given by the objects with the trivial right $B$-action. The category  $ \Modz{A}{B} $  is  equipped with a monoidal structure given by the tensor product $\otimes_B$. Moreover, with respect to the obvious monoidal structure on  $\Modz{H}{}$.  Takeuchi’s equivalence is readily endowed with the structure of a monoidal equivalence (see Section 4 of \cite{MMF2}). 

Now let~$M$ be an object of~$\Modz{A}{B}$. Then for any $m \in M$ and $b \in B$, we have
\begin{eqnarray*}
 m b &=& U^{-1} (U (m b ))\\
     &=& U^{-1} (m_{(-1)} b_{(1)} \otimes [ m_{(0)} ] \epsilon (b_{(2)}))\\
     &=& m_{(- 2)} b S (m_{(- 1)}) m_{(0)},
\end{eqnarray*}
where we have used~\eqref{6thjune242}, \eqref{6thjune243} and that the right $B$-module structure of~$\Phi(M)$ is trivial.

Thus, if $M$ is an~object of~$\Modz{A}{B}$, then
\begin{equation} \label{6thdec231}
m b := m_{(- 2)} b S (m_{(- 1)}) m_{(0)},
\end{equation}
for $m \in M$ and $b \in B$.

\begin{rem}
  The category $\modz{A}{B}$ of Definition~\ref{10thjune241} is then the full subcategory of $\Modz{A}{B}$ whose objects are finitely generated as left $B$-modules and Theorem~\ref{3rdapril241} shows that $\Phi$ is a~monoidal equivalence between $\modz{A}{B}$ and $\lmod{H}{}$, where the latter is the category of finite dimensional left $H$-comodules. 
\end{rem}

Then we have the following result:

\begin{prop} \label{20thdec231}
Suppose~$M$ and~$N$ are objects in~$\modz{A}{B}$ and $\phi: M \rightarrow N$ is a~left $B$-linear and left $A$-comodule map. Then~$\phi$ is automatically a~morphism of the category $\modz{A}{B}$, i.e, $\phi$ is right $B$-linear.
\end{prop}
\begin{proof} Indeed, if $x \in M$ and $b \in B$, then by applying \eqref{6thdec231} twice, we obtain
\begin{eqnarray*}
  \phi(x  b) &=& \phi(x_{(-2)}bS(x_{(-1)}) x_{(0)}) =  x_{(-2)}bS(x_{(-1)}) \phi(x_{(0)})\quad
                 \text{(as $\phi$ is left B-linear)}\\
    &=& \phi (x)_{(-2)} b S(\phi (x)_{(-1)}) \phi(x)_{(0)}    = \phi(x) b
\end{eqnarray*}
 by the left $A$-colinearity of $\phi$. 
\end{proof} 

We end this subsection with the following corollary:

\begin{cor} \label{23rdjuly241}
Suppose $B = A^{\co(H)}$ is a quantum homogeneous space equipped with an $A$-covariant differential calculus such that $\Omega^1$ is an object of the category $\modz{A}{B}$. Then $ {}^{\co(A)} (\Omega^1 \otimes_B \Omega^1 \otimes_B \Omega^1) = 0 $ if and only if the only morphism from $\mathbb{C}$ to $\Phi (\Omega^1) \otimes \Phi (\Omega^1) \otimes \Phi (\Omega^1) $ in $\lmod{H}{}$ is the zero morphism. 
\end{cor} 
\begin{proof}
This follows by observing that there is a one to one correspondence between $ {}^{\co(A)} (\Omega^1 \otimes_B \Omega^1 \otimes_B \Omega^1) $ and the set of all left $B$-linear left $A$-colinear maps from $B$ to  $ \Omega^1 \otimes_B \Omega^1 \otimes_B \Omega^1$, defined by $X \mapsto f_X$, where 
$f_X: B \rightarrow \Omega^1 \otimes_B \Omega^1 \otimes_B \Omega^1$
is defined as $f_X (b) = b X. $ As $f_X$ is left $B$-linear and left $A$-colinear,  Proposition \ref{20thdec231} implies that $f_X$ is a morphism in $\modz{A}{B}$. Now the result follows from Takeuchi's equivalence. 
\end{proof}

\subsection{Some results on bar-categories}

\begin{prop} \label{3rdjuly243}
If $B = A^{\co(H)}$ is a quantum homogeneous space of a Hopf $\ast$-algebra $A$, then the category $\modz{A}{B}$ is a bar category. 
\end{prop}
\begin{proof} To begin with, let $\Modz{A}{B}$ denote the monoidal category introduced in Subsection \ref{3rdjuly241}. If $M$ is an object of $\Modz{A}{B}$, then $MB^+=B^+M$. Using this, it is easy to see that $ \overline{M} B^+ = B^+ \overline{M} $ and that $\Modz{A}{B}$ is a bar category. Now if  $M$ be an object of $\modz{A}{B}$, then $M$ is an object of $\Modz{A}{B}. $ Hence, $\overline{M}$, equipped  with the bimodule structure defined in \eqref{28thnov231} and comodule structure defined in \eqref{25thnov232} is an object of $\Modz{A}{B}$. Then by \eqref{6thdec231}, we get 
\begin{equation} \label{3rdapril244}
    \overline{m}\cdot_{\overline{M}}b =  \left (m_{(-2)}^* b S(m_{(-1)}^*)\right) \cdot_{\overline{M}}\overline{m_{(0)}}.
\end{equation}

\noindent As $M$ is finitely generated as a left $B$-module, it is easy to see from \eqref{3rdapril244} that $\overline{M}$ is also finitely generated as a left $B$-module. Thus, $\overline{M}$ is indeed an object of $\modz{A}{B}$. The rest of the properties of a bar-category can be easily verified.
\end{proof}

\begin{lem} \label{27thmarch243}
    Suppose $M$ is an object in the category $\modz{A}{B}$, where $B\subseteq A$ is a quantum homogeneous space. Let $\Phi$ denote the Takeuchi's equivalence, then for any object $M$ in $\modz{A}{B}$, $\Phi(\overline{M})\cong \overline{\Phi(M)}$.
\end{lem}
\begin{proof}
  $M$ is an object in the category $\modz{A}{B}$ and so  $\overline{M}$ is an object in $\modz{A}{B}$ by Proposition \ref{3rdjuly243}. Therefore, $\overline{M}$ is an object of $\Modz{A}{B}$ and hence,  
  $$\overline{M}B^+ =  B^+\overline{M}.$$ 
  We are going to use this fact in the proof.
    
    Suppose $V= \frac{\overline{M}}{B^+\overline{M}} $
and $W=\overline{\frac{M}{B^+M}}$.
Define the surjective map
\begin{equation} \label{5thjuly241}
f: \overline{M} \to W \quad \text{by} \quad \overline{m}\mapsto \overline{[m]},
\end{equation}
where $[m]$ denotes the equivalence class of $m$ in $\Phi (M). $

We will prove that $f$ induces an isomorphism from $V$ to $W$ in $\lmod{H}{}$. Indeed, if $\overline{m}\in\ker(f)$, then  $[m]=0$. Thus,  we have $m \in B^+M$, so that  $\overline{m} \in \overline{M}B^+ =  B^+\overline{M}$. Hence we have that $\ker(f)\subseteq B^+\overline{M}$.

Conversely, if  $\sum_{i=1}^n b_i \overline{x_i}\in B^+\overline{M}$, then 
$$f(\sum_{i=1}^n b_i \overline{x_i})=f(\sum_{i=1}^n  \overline{x_ib_i^{\ast}})=\overline{[\sum_{i=1}^n {x_i}b_i^{\ast}]}.$$
Now the element  $\sum_{i=1}^n {x_i}b_i^{\ast}\in MB^+ =  B^+ M$ which implies that  
$$ f (\sum^n_{i = 1} b_i \overline{x_i}) = \overline{[\sum_{i=1}^n {x_i}b_i^{\ast}]}=0.$$
This proves that $\ker (f) = B^+ \overline{M} $ and 
 we have a vector space isomorphism denoted by $f$ (by an abuse of notation) from $V$ to $W$ defined as $f[\overline{m}]= \overline{[m]}$. 
 Finally, the $H$-covariance of the map $f$ can be checked from the following computation:
 \begin{align*}
     ^W\delta(\overline{[m]})&= \pi(m_{(-1)})^{\ast} \otimes \overline{[m_{(0)}]}\\
    & = (\id\otimes f) (\pi(m_{(-1)})^{\ast} \otimes [\overline{m_{(0)}}])\\
    & = (\id\otimes f) ^V\delta(\overline{[m]}).
 \end{align*}
\end{proof}

We end this subsection with a couple of simple observations which have been used in the article. 

\begin{prop}\label{25thnov235}
Suppose $(\Omega^{\bullet}, \wedge, \partial, \overline{\partial})$ is a complex structure on a $\ast$-algebra $B$ and $\mathcal{E}$ be a $B$-bimodule. Let $ \beta: \overline{\mathcal{E}}\to \overline{\mathcal{E} } \otimes_B \Omega^{(0,1)}$
 be a $\mathbb{C}$-linear map.

 If $\beta(\overline{e})= \sum_{i=1}^k\overline{f_i} \otimes_B \omega_i$, then
   $$ \widehat{\beta}: \mathcal{E } \to \Omega^{(1,0)}\otimes_B \mathcal{E} ~ \text{defined as $\widehat{\beta}(e)= \sum^k_{i = 1} \omega_i^{\ast} \otimes_B f_i$} $$
is well defined.
\end{prop}
\begin{proof}  We work in the bar-category $\qMod{}{B}{}{B}$.
We note that 
\begin{align*}
    \widehat{\beta}(e)= \mathrm{bb}^{-1}_{\Omega^1 \otimes_B \mathcal{E}} \overline{\Upsilon_{\Omega^{1},\mathcal{E}}^{-1}(\id\otimes_B \star_{\Omega^1})(\beta(\overline{e}))}.
\end{align*}
Therefore, $\widehat{\beta}$ is well-defined. 
\end{proof}

We will also need the following lemma:

\begin{lem}\label{24thnov231}
Suppose $A$ is a Hopf $\ast$-algebra and $\mathcal{F}$ is an object in the bar category $\lMod{A}{}$. If $a_i, b_j$ are finitely many elements in $A$ and $x_i, y_j \in \mathcal{F}$ are such that  

$\sum_i a_i \otimes \overline{x_i}= \sum_j b_j \otimes \overline{y_j} $ in $A\otimes \overline{\mathcal{F}},$.  then $\sum_i a_i^{\ast} \otimes {x_i}= \sum_j b_j^{\ast} \otimes {y_j}$.
\end{lem}
\begin{proof} The proof follows from the following computation:
\begin{align*}
    \sum_i a_i^{\ast} \otimes {x_i} & = \mathrm{flip}\left(\sum_i x_i \otimes a_i^{\ast} \right) = \mathrm{flip} \circ (\mathrm{bb}^{-1}_{\mathcal{F}}\otimes \star_A^{-1}) \left(\sum_i \overline{\overline{ x_i}}\otimes \overline{a_i} \right)\\
    & = \mathrm{flip} \circ (\mathrm{bb}^{-1}_{\mathcal{F}}\otimes \star_A^{-1}) \circ \Upsilon_{A,\overline{\mathcal{F}}}\left(\overline{ \sum_i a_i \otimes \overline{x_i}}\right)\\
    &= \mathrm{flip} \circ (\mathrm{bb}^{-1}_{\mathcal{F}}\otimes \star_A^{-1}) \circ \Upsilon_{A,\overline{\mathcal{F}}} \left(\overline{\sum_j b_j \otimes \overline{y_j}}\right)
     = \sum_j b_j^{\ast} \otimes y_j.
\end{align*}
\end{proof}

\bibliographystyle{amsalpha-ak}
\bibliography{bigbib}

\end{document}